\documentclass[a4paper,10pt,review]{article}
\usepackage{amssymb,amsmath, amsthm}
\usepackage[]{graphicx} 
\usepackage{stmaryrd}
\usepackage{tikz}
\usepackage{algorithm,algorithmic}
\usepackage[]{units}
\usepackage{multirow}
\usepackage{rotating}
\usepackage{booktabs}

\newcommand\hb[1]{\mathop{%
  \tikz[baseline=(n.base)]{\node(n)[inner sep=1pt, outer sep=0pt]{$#1$};%
    \draw[line cap=round](n.west)--(n.south west)--(n.south east)--(n.east);}}  }

\newcommand\lhb[1]{\mathop{%
  \tikz[baseline=(n.base)]{\node(n)[inner sep=1pt, outer sep=0pt]{$#1$};%
    \draw[line cap=round](n.south west)--(n.south east)--(n.east); }}}

\newcommand\rhb[1]{\mathop{%
  \tikz[baseline=(n.base)]{\node(n)[inner sep=1pt, outer sep=0pt]{$#1$};%
    \draw[line cap=round](n.west)--(n.south west)--(n.south east); }}}

    \newcommand{\deemph}[1]{{\color{black!40}#1}}

\newtheorem{theorem}{Theorem}[section]
\newtheorem{cor}[theorem]{Corollary}
\newtheorem{conj}[theorem]{Conjecture}
\newtheorem{lem}[theorem]{Lemma}
\newtheorem{de}[theorem]{Definition}
\newtheorem{ex}[theorem]{Example}

\DeclareMathOperator{\Mod}{Mod}
\newtheoremstyle{dotless}{}{}{\itshape}{}{\bfseries}{}{ }{}
\theoremstyle{dotless}
\newtheorem*{rem}{Remark}
\DeclareMathOperator{\pf}{pref}
\DeclareMathOperator{\suf}{suff}
\DeclareMathOperator{\pat}{pat}
\DeclareMathOperator{\vect}{tup}
\DeclareMathOperator{\class}{class}
\DeclareMathOperator{\rhs}{rhs}
\DeclareMathOperator{\MERF}{MERF}

\newtheorem{innercustomthm}{Theorem}
 \theoremstyle{dotless}
\newenvironment{customthm}[1]
  {\renewcommand\theinnercustomthm{#1}\innercustomthm}
  {\endinnercustomthm}

\begin{document}

 \title{On the $2$-Abelian Complexity of the Thue--Morse Word}
\author{Florian Greinecker \\  \texttt{greinecker@math.tugraz.at} \\ \small{\textit{ Department of Analysis and Computational Number Theory (Math A),}}\\
\small{\textit{Steyrergasse 30, 8010 Graz, Austria}}}
%

%
%

\maketitle

\begin{abstract}
We show that the $2$-abelian complexity  of the infinite Thue--Morse word is $2$\hbox{-}regular, and  other properties of the $2$-abelian complexity, most notably that it is a
concatenation  of palindromes of increasing length. We also
show sharp bounds for the length of unique extensions of factors of size $n$, occurring in the Thue--Morse word.


\end{abstract}

\section{Introduction}

This paper contributes to the study of the Thue-Morse word. The infinite Thue--Morse word $t=01101001100101101001011\cdots$
is defined as \[t:=\lim_{n \rightarrow \infty}m^n(0)\] where $m$ is the morphism
\[ m\colon 0 \mapsto 01, 1 \mapsto 10.\]
The set of all finite factors of the Thue--Morse word will be denoted by $T$, while $T_n$ stands for the set of Thue--Morse factors of length $n$.

In this paper we will prove three theorems about $t$. The first one is about extensions of factors of $t$ . If we want to prolongate a factor of $t$ to a longer factor
of $t$ there is sometimes
only one possible letter. For example after $00$ the next letter
has to be a $1$. We will give upper and lower bounds for the length of such extensions in Theorem \ref{rf}.

Then we take a look at the $2$-abelian complexity sequence of the Thue--Morse word. We will prove  that it is a concatenation
of palindromes of increasing length (Theorem \ref{pal}) and secondly the following theorem.

\begin{theorem} \label{main}
 The $2$-abelian complexity of the Thue--Morse word $t$ is $2$-regular.
\end{theorem}

To understand this theorem let us take a look at the  first concept in Theorem~\ref{main}: the $\ell$-abelian complexity.

 The $\ell$-abelian complexity is a complexity measure, which was first introduced in 1981 by Karhum\"aki \cite{Karhumaeki1981}. 
  The $\ell$-abelian complexity 
$\mathcal{P}_w^{(\ell)}(n)$ of an infinite word $w$ builds a bridge between the  abelian complexity which corresponds to $\ell=1$  and the factor complexity 
 which corresponds to $\ell=+\infty$ 
and allows
a finer resolution. The abelian complexity $\mathcal{P}_w^{(1)}(n)$ of an infinite word counts the anagrams of length $n$ while the
factor complexity $\mathcal{P}_w^{(\infty)}(n)$ counts the factors in $T_n$.
 The abelian complexity sequence of the infinite Thue--Morse word  is
 $(\mathcal{P}_t^{(1)}(n))_{n\geq 0}=1,(2,3)^\omega$.
The  factor complexity  \cite{Brlek} of the Thue--Morse word $t$  is well known
\[\mathcal{P}_t^{(\infty)}(0)=1, \quad \mathcal{P}_t^{(\infty)}(1)=2, \quad \mathcal{P}_t^{(\infty)}(2)=4, \]
\[\mathcal{P}_t^{(\infty)}(n)=\begin{cases}4n-2\cdot 2^m-4, \text{ if } 2\cdot 2^m <n\leq 3\cdot 2^m; \\
  2n+4\cdot 2^m-2, \text{ if } 3\cdot 2^m <n\leq 4\cdot 2^m. 
  \end{cases}
\]

Before we define  $\ell$-abelian complexity we need some vocabulary. For a word $w=w_0w_1\cdots w_n$ 
the prefix of length $\ell$ is defined as $\pf_\ell(w):=w_0\cdots w_{\ell-1}$ while the suffix of length $\ell$
is $\suf_\ell(w):=w_{n-\ell+1}\cdots w_{n}$.

We write $|w|$ to denote the \emph{length} of a word $w$. If $v$ is a factor of $w$ the number of occurrences of $v$ in $w$ is denoted by $|w|_v$. We write
$\mathbb{N}_0$ for the natural numbers, including $0$.

\begin{de} \label{kabel}
 For an integer $\ell \geq 1$, two words $u,v\in A^*$, for some alphabet $A$, are \emph{$\ell$-abelian equivalent} if
 \begin{enumerate}
  \item $\pf_ {\ell-1}(u)=\pf_ {\ell-1}(v)$ and $\suf_ {\ell-1}(u)=\suf_ {\ell-1}(v)$, and \label{rrr}
  \item for all $w \in A^*$ with $|w|=\ell$ the number of occurrences of $w$ in $u$ and $v$ is equal, i.e. $|u|_w=|v|_w$.
 \end{enumerate}
We then write $u\equiv_\ell v$. 
\end{de}

There are several equivalent definitions of $\ell$-abelian equivalence (cf. \cite{Karhumaki2013a}), we use the one from \cite{EliseLeiden}. Note that the given
definition is not minimal, it would suffice to use either $\pf_ {\ell-1}(v)$ or  $\suf_ {\ell-1}(v)$.

It is easy to check that $\ell$-abelian equivalence is indeed an equivalence relation. The first part of the definition, where we fix the prefix and suffix,  guarantees that two $\ell$-abelian equivalent words are also
$(\ell-1)$-abelian equivalent.

\begin{ex}
 Let us take two words $w=001011$ and $v=001101$. We see that $w\equiv_2 v$ since $|w|_{00}=1$, $|w|_{01}=2$, $|w|_{10}=1$, $|w|_{11}=1$ and we get the same values for $v$. 
Furthermore both words have the same prefix and suffix. On the other hand $w\not\equiv_3 v$ since
 $|001011|_{010}=1$ and  $|001101|_{010}=0$. Also the suffixes differ, $11\not=01$.
\end{ex}
 
 Since $\equiv_\ell$ is an equivalence relation it is natural to count equivalence classes.
 
We are  interested in the number of $2$-abelian equivalence classes for words of a given length:
 
 \[\mathcal{P}_t^{(2)}(n):= \#(T_n/_{\equiv_2}).\]
  Usually we write $\mathcal{P}_w^{(\ell)}(n)$  to denote the number of $\ell$-abelian equivalence classes of factors of $w$ of length $n$, where $w$ is an infinite word. 
 In rest of this paper, we will only consider  the $2$-abelian complexity of the  Thue--Morse word $t$. Therefore we will use the simpler notation 
$\mathcal{P}_n:=\mathcal{P}_t^{(2)}(n)$.

 The sequence starts with 
 \begin{equation*}
  \begin{split}
(\mathcal{P}_n)_{n\geq 0}=1,2,4,6,8,6,8,10,8,6,8,8,10,10,10,8,8,6,8,10,10,8,10,12,12, \\
10,12,12,10,8,10,10,8,6,8,8,10,10,12,12,10,8,10,12,14,12,12,12,12,\dots.
 \end{split}
\end{equation*}

 \begin{de} \label{con}
  We assign to every word $w$  its equivalence class. To denote the $2$-abelian equivalence class of a word $w$ we use a $6$-tuple.
\[ \class\colon  T \rightarrow \mathbb{N}_0^4 \times \{0,1\}^2,\enskip w \mapsto
(|w|_{00}, |w|_{01}, |w|_{10}, |w|_{11}, w_0, w_n).\]
 \end{de}

       \begin{ex}\label{rr}
        We have $\class(w)=(1,2,3,1,1,0)$ for $w=10011010$.
       \end{ex}
       
Karhum\"aki, Saarela and Zamboni showed in \cite{Karhumaeki2013} that for $n\geq 1$, $m\geq 0$ we have
\[\mathcal{P}_t^{(2)}(n)=O(\log n), \quad \mathcal{P}_t^{(2)}((2\cdot 4^m+4)/3)=\Theta(m) \quad \text{ and }\mathcal{P}_t^{(2)}(2^m+1)\leq 8.\]
Actually, $\mathcal{P}_t^{(2)}(2^m+1)=6$, for $m\geq 1$.

Theorem~\ref{main} combines two concepts: $\ell$-abelian complexity and $k$-regular sequences.
We just treated $\ell$-abelian complexity, let  us now look at $k$-regular sequences.

Allouche and Shallit introduced $k$-regular sequences in 1990 \cite{1990}. It is a well-known theorem by Eilenberg
\cite{Eilenberg} that a sequence is $k$-automatic if and only if its $k$-kernel is finite.

\begin{de}
 Let $k \geq 2$ be an integer. The \emph{$k$-kernel} of a sequence $(a(n))_{n\geq 0}$ is the set of subsequences
 \[\{(a(k^en+c))_{n\geq 0}\mid e\geq0, \enskip 0\leq c< k^e\}.\]
\end{de}
 For example, the Thue-Morse sequence is $2$-automatic.
Allouche and Shallit \cite{1990} took this characterization of $k$-automatic sequences via the kernel and extended it to $k$-regular sequences.
\begin{de}[Allouche and Shallit]
Let $k\geq 2$ be an integer. An integer sequence $(a(n))_{n\geq 0}$ is \emph{$k$-regular} if the $\mathbb{Z}$-module generated by its $k$-kernel is finitely generated.
\end{de}

Just recently research begun to investigate the regularity of the abelian complexity. Madill and Rampersad showed that the abelian complexity of the paperfolding
 word is $2$-regular \cite{Madill2013}.

 This article solves
an open conjecture from Elise Vandomme, Aline Parreau and Michel Rigo \cite{EliseLeiden} who conjectured that the $2$-abelian complexity of the infinite Thue--Morse word is 
$2$-regular. This is a special case of a more general conjecture by Rigo.

\begin{conj}
The $2$-abelian complexity of any $k$-automatic word is an $k$\hbox{-}regular sequence.
\end{conj}

Shortly after the discovery of our proof, they found an independent proof \cite{Rigo} of their own,
which uses the palindromic structure of the sequence.

This paper is  organized as follows: 

After some definitions in the rest of this section we will introduce \emph{reading frames} in Section~\ref{frames}. Reading frames are a  factorization of words into 
factors $v_i$ of the form $v_i=m^q(0)$ or $v_i=m^q(1)$ for some $q\in\mathbb{N}$, plus a prefix and suffix of shorter length.
Reading frames are a  natural
way to think about the Thue--Morse word since they preserve the morphism structure.

We use these reading frames in Section~\ref{merf}  to prove Theorem~\ref{rf} about unique extensions of Thue--Morse factors.
For a factor $w$ in $t$ there is sometimes only one possibility
for the next (or previous) letters $x_1\cdots x_n$ so that $wx_1\cdots x_n$ (or $x_1\cdots x_nw$) is again a factor of $t$. We give lower and upper bounds 
for the lengths of such unique extensions. Section~\ref{merf}  can be skipped if one is only interested
in the proof of the $2$-regularity.

In  Definition \ref{con} we needed $6$ values to describe  the  $2$-abelian equivalence class of a factor $w$, $4$~binary values and $2$~integer values.
We
introduce the \emph{odd frame} in Section~\ref{obf} in order to simplify this $6$-tuple of values.
The odd frame is a shifted reading frame, which  does not preserve the morphism structure but allows 
to use only $3$ values to represent the $2$-abelian equivalence
class of a factor $w$ with $2$ binary values and $1$ integer value. Only the possible values of the integer $p(w)$, the possible numbers of pairs in a factor $w$, are
nontrivial to determine.
  
Beside the odd frame we also introduce a \emph{short coding}. The short coding is a way to  encode words  in the odd frame so that the numbers of pairs in a factor $w$ can
be seen on the first view.

  In Section~\ref{pa} we use the properties
of pure odd words to prove a recursion (Theorem~\ref{pairs}) on two types of sets, where 
$\mathbf{pairs}(n)$ is the set of all  possible values of $p(w)$ for factors $w$ of length $n$. Once we have the recursion we can use it to determine $\mathcal{P}_t^{(2)}(n)$
for all $n$  (Theorem~\ref{conclass}). These two theorems are used in all further proofs.

Equipped with Theorem~\ref{pairs} and Lemma~\ref{conclass} we prove the main Theorem~\ref{main}  in Section~\ref{2reg} by showing $13$ linear relations given in Theorem 1.1.*.
For each of the $13$ relations the calculations are similar but, since we have to look at three cases
for each of them, a bit lengthy.

Finally we show  some additional  properties of $\mathcal{P}_t^{(2)}(n)$ in Section~\ref{eigen}, most notably
that $\mathcal{P}_t^{(2)}(n)$ is 
a concatenation of longer and longer palindromes. Again we use Theorem~\ref{pairs} and Theorem~\ref{conclass} to do this. We also show that $\mathcal{P}_t^{(2)}(n)$ is unbounded.

Before we continue with the proof we will need some definitions. We will use the fact that $t$ is overlap-free (cf. \cite{automatic}, Theorem 1.6.1). An \emph{overlap} is a word of the
form $xwxwx$, where $w$ is a word, possibly empty,
and $x$ is a single letter. The word $t$ is also cubefree, i.e. it contains no word of the form $www$ where $w$ is a nonempty factor.

The Thue--Morse word is defined over the binary alphabet $A=\{0,1\}$. A (literal) \emph{pattern} is a word over the alphabet $E=\{\alpha,\beta\}$.
Furthermore we define the involutive morphisms $\overline{0}:=1$ and $\overline{1}:=0$ and  similarly $\overline{\alpha}:=\beta$ and $\overline{\beta}:=\alpha$. If $w$ is the 
word $w=x_1x_2\dots x_n$ we call the word $\overline{w}:=\overline{x}_1\overline{x}_2\dots\overline{x}_n$ the \emph{complement} of $w$. An \emph{assignment} of a pattern
is the image of a pattern under a bijective function from $E$ to $A$. A pattern $p$ and a word $w$ are \emph{equal} ($p=w$) if $p=w$ for one assignment of $p$.
We introduce patterns  to avoid case distinctions for complementary words.
\begin{ex}
 The pattern $\alpha\overline{\alpha}\alpha\overline{\alpha}\overline{\alpha}$ can stand for $01011$ or $10100$ depending on the assignment of $\alpha$. So does $\overline{\alpha}\alpha\overline{\alpha}\alpha\alpha$.
\end{ex}

If we want a pattern out of a word, we do this via the morphism
\[ \pat\colon 0 \mapsto \alpha,\, 1 \mapsto \overline{\alpha}.\]

 From now on we will just write \emph{word} if we mean a finite factor of $t$. 
 
\section{Reading frames}\label{frames}

From its definition via the morphism $m$ it is clear that the Thue--Morse word is composed of copies of its first $2^q$ letters,  $q\in\mathbb{N}$, and their complements. 
To denote the special role
of these words we define $f_{2^q}:=t_0 t_1 \cdots t_{2^q-1}$.
\begin{ex}
 We take $q=2$ and get $\pat(f_{2^2})=\alpha\overline{\alpha}\overline{\alpha}\alpha$ which gives us
 \[t=01101001100101101001011\cdots=\hb{0110}\hb{1001}\hb{1001}\hb{0110}\hb{1001}011\cdots\]
\end{ex}
A property that will be  useful later is that the word $f_{2^q}$ has the image $m(f_{2^q})=f_{2^{q+1}}$ and the preimage $m^{-1}(f_{2^q})=f_{2^{q-1}}$.

\begin{de}
A \emph{$2^q$-reading frame} of a word $w\in T$ is a factorization of $w$ into  words $w=pv_1\cdots v_ms$, where $v_1,\cdots,v_m$ are words of length $2^q$ plus a 
prefix $p$ and a  suffix $s$ with $|p|,|s|<2^q$, so that $v_i$ is a word with $v_i=\pat(f_{2^q})$, for the prefix $p$ we have $p=\pat(\suf_{|p|}(f_{2^q}))$ 
and  for the suffix $s$ we have accordingly  $s=\pat(\pf_{|s|}(f_{2^q}))$. 
\end{de}
 We call $p,s$ and the $v_i$ \emph{frame words}, especially
the $v_i$ are called \emph{complete frame words}.
The $1$-reading frame is called the \emph{trivial frame}. A word $w\in T$ may have several different $2^q$-reading frames but at most $2^q$. If we shift a $2^q$\hbox{-}reading frame
one letter to the right we are in a new reading frame and after $2^q$ shifts we are in the original reading frame again. If there is only one $2^q$-reading frame it is called 
the \emph{extensible} reading frame.

We said in the beginning of this section that $t$ is composed of copies of its first $2^q$ letters and their complements. So the infinite Thue--Morse word $t$
has a $2^q$-reading frame for every $q$. Since any word $w$ in $T$ is a factor of $t$ it can be read in the same reading frame as $t$. And if
$w$ has only one $2^q$-reading frame it has to be the $2^q$-reading frame of $t$.

We can can get the previous and next letters of $t$ if we fill up prefix and suffix to complete frame words in the $2^q$-reading 
frame. We will  use a $\deemph{gray}$ font for filled up letters. Since $t$ is infinite the extensible $2^q$-reading frame of $w$ can be extended to arbitrary length (possibly
in different ways) but the filled up letters are unambiguous. So if the filled up letters in the prefix and suffix  give a word which is not in $T$ we can not be in the extensible reading frame.

If there is an extensible $2^q$-reading frame, we call the $2^{q-1}$-reading frame that we get by splitting every complete $2^q$-reading frame word into two complete $2^{q-1}$-reading frame words, extensible
too. 

\begin{ex}\label{frame}
 The word $0101$ has two $2$-reading frames: $\hb{01}\hb{01}$ and $\lhb{0}\hb{10}\rhb{1}$, but the extensible $2$-reading frame is
 $\hb{01}\hb{01}$.
 We can not extend the $\lhb{0}\hb{10}\rhb{1}$
 reading frame since if we fill up the letters we  get $\hb{\deemph{1} 0}\hb{10}\hb{1\deemph{0}}$, but $t$ is cubefree.
 
 And $0101$ has a unique $4$-reading frame: $\lhb{01}\rhb{01}$. We can get the $4$-reading frame by merging two $2$-reading frame words.  Since every complete $4$-reading 
 frame word has the pattern $\alpha\overline{\alpha\alpha}\alpha$ it can not be $\hb{0101}$ so the extensible $4$-reading frame is $\hb{\deemph{10}01}\hb{01\deemph{10}}$. Thus $0101$ is extensible.
\end{ex}

Let us state some of the previous comments explicitly as lemma.

\begin{lem}
 A word $w\in T$ has a $2^q$-reading frame for every $q\in\mathbb{N}_0$.
\end{lem}
\begin{proof}
 By definition $w$ occurs somewhere in $t$. But $t$ can be read in a $2^q$-reading frame for every $q$ therefore $w$ can be read in a $2^q$-reading frame too.
\end{proof}
\begin{cor}
 A word $w$ is not in $T$ if there exists an integer $q$, so that $w$ has no $2^q$-reading frame.
\end{cor}

\begin{ex}
 The words assigned to the patterns $\alpha\alpha\alpha$, $\alpha\alpha\overline{\alpha}\alpha\alpha$ and  $\alpha\overline{\alpha}\alpha\overline{\alpha}\alpha$ are not 
 in $T$ since $t$ is overlap-free. We can also prove this with reading frames.
 The the words assigned to the patterns $\alpha\alpha\alpha$ and $\alpha\alpha\overline{\alpha}\alpha\alpha$ are not in $T$ since they have no $2$-reading frame. 
 The the words assigned to the pattern $\alpha\overline{\alpha}\alpha\overline{\alpha}\alpha$ are not in $T$ since they have no $4$-reading frame.

\end{ex}

\section{Maximal extensible reading frames}\label{merf}
 An  \emph{extension} of a word $w \in T$ is a pair of words $(v_1,v_2)$ with 
$v_1,v_2\in T$, $|v_1|+|v_2|>0$ so that $v_1wv_2 \in T$. An extension $(v_1,v_2)$ is \emph{unique} if for all pairs $(y_1,y_2)\in\{0,1\}^*\times \{0,1\}^*$ with
$|v_1|=|y_1|, |v_2|=|y_2|$ and $(v_1,v_2)\not=(y_1,y_2)$ it follows  $y_1wy_2 \not\in T$.

We already saw that we can get unique extensions of a word $w \in T $ if we fill up prefix and suffix to complete frame words in the extensible reading frame. In this
section we will prove that we get all unique extensions of a word $w\in T$ by filling up  prefix and suffix in a certain $2^q$-reading frame.

There is a \emph{maximal extensible reading frame} (abbreviated as MERF), since if $w$ is a factor of $f_{2^q}$ it can not uniquely determine
a $2^{q+1}$-reading frame.
\begin{ex}
 The word $w=011$  has the extensible $2$-reading frame $\hb{01}\hb{1\deemph{0}}$. The $2$-reading frame is the MERF since there are two possible $4$-reading frames;
$w$ is a factor of $\hb{0110}$ and $w$ is also a factor of $\hb{1001}\hb{1001}$.
\end{ex}

For small cases it is easy to determine the maximal extensible reading frame by hand, while longer words can be reduced to the short cases as preimages
under the morphism $m$.

\begin{table}[ht]
\begin{center}
 \begin{tabular}{cc}
\toprule
 Pattern & MERF   \\
\midrule
$\hb{\alpha}$ & $1$-reading frame \\
$\lhb{\alpha}\rhb{\alpha}$ & $2$-reading frame \\
$\hb{\alpha}\hb{\overline{\alpha}}$& $1$-reading frame  \\
$\lhb{\alpha}\hb{\alpha\overline{\alpha}}$& $2$-reading frame \\
$\hb{\alpha}\hb{\overline{\alpha}}\hb{\alpha}$ & $1$-reading frame\\
$\hb{\alpha\overline{\alpha}}\rhb{\overline{\alpha}}$ & $2$-reading frame\\
 \bottomrule
 \end{tabular}
\hspace{3em}
  \begin{tabular}{cccc}
\toprule
Pattern & MERF \\
\midrule
 $\lhb{\alpha\alpha\overline{\alpha}} \rhb{\alpha}$& $4$-reading frame\\
  $\lhb{\alpha}\hb{\alpha\overline{\alpha}} \rhb{\overline{\alpha}}$ & $2$-reading frame\\
 $\lhb{\alpha}\rhb{\overline{\alpha}\alpha\alpha}$ &$4$-reading frame\\
$\lhb{\alpha\overline{\alpha}}\rhb{\alpha\overline{\alpha}}$ &$4$-reading frame\\
$\hb{\alpha\overline{\alpha}}\hb{\overline{\alpha}\alpha}$ &$2$-reading frame\\
\\
 \bottomrule
 \end{tabular}
 \caption{MERFs for all nonempty words in $T$ up to length $4$.}\label{lookup}
\end{center}
\end{table}
The extensible reading frame of a factor $v$ of $w$ also determines the extensible reading frame of $w$. Therefore every word
in $T$ of length at least $4$ has an extensible $2$-reading frame. We can now formulate an algorithm to determine the 
MERF of a word $w$. 

The algorithm determines the extensible $2$-reading frame of the word and fills 
 up the prefix and the suffix of $w$ to complete frame words, then takes the preimage of the new word and repeats those steps till  it reaches a word with no extensible
 $2$-reading frame. In every step the reading frame size doubles and the algorithm will need $q$ steps if the MERF has size $2^q$.
 
 In every step there will be at most two new letters before the word size is halved. So the words will get shorter in every step until they have a length of $4$ or shorter.
 Since the algorithm terminates for all words in Table~\ref{lookup} it will terminate in general.
 
\begin{algorithm}
 \caption{Determines the MERF of a word $w\in T$ and fills the MERF}
 \begin{algorithmic}
   \STATE $\textbf{procedure: } MERF(w)$
   \STATE $q \leftarrow 0$
   \STATE $w' \leftarrow w$
   \WHILE{ $w'$ has a nontrivial reading frame }
   \STATE $q \leftarrow q+1$
   \STATE $w' \leftarrow \textbf{FillFrame}(w')$ \COMMENT{Determines and fills the extensible $2$-reading frame}
      \STATE $w' \leftarrow m^{-1}(w')$
   \ENDWHILE
   \RETURN $m^q(w'),\enskip 2^q\text{ ``-reading frame.''}$
 \end{algorithmic}
\end{algorithm}

To decide whether $w$ has a nontrivial reading frame we can use Table~\ref{lookup} as lookup table, since there are only $6$ words ($3$ patterns) with 
a trivial reading frame. For $\textbf{FillFrame}(w')$ we use the same lookup table at the first $4$ letters of $w'$ to determine the $2$-reading frame, then we find the frame prefix and
suffix and fill them up. The original word $w$ will occur only once as factor in $m^q(w')$. Let us look at an example.

\begin{ex}
What is $\MERF(0110010)$?
\newline
We start with $w=\hb{01}\hb{10}\hb{01}\rhb{0}$ and $q=0$. Then we enter the while loop  and get
  \[q=1,\enskip FillFrame(w')=\hb{01}\hb{10}\hb{01}\hb{0\deemph{1}}, \enskip m^{-1}(w')=\lhb{0}\hb{10}\rhb{0},\] 
\[q=2,\enskip FillFrame(w')=\hb{\deemph{1}0}\hb{10}\hb{0\deemph{1}},{\color{white}\hb{00}}\enskip m^{-1}(w')=\lhb{1}\hb{10},{\color{white}\rhb{0}}\]
\[q=3,\enskip FillFrame(w')=\hb{\deemph{0}1}\hb{10},{\color{white}\hb{00}}{\color{white}\hb{00}} \enskip m^{-1}(w')=\hb{0}\hb{1}.{\color{white}\hb{00}}\]
Since $\hb{0}\hb{1}$ has a trivial reading frame we leave the while loop and the algorithm returns
$\MERF(0110010)\colon \hb{\deemph{011010}01}\hb{10010\deemph{110}}\text{ $8$-reading frame}$.
 \end{ex}

As a consequence of the algorithm we have the following lemma:
\begin{lem}\label{wl}
A factor $w$ of the Thue--Morse word of length $2^q\leq|w|<2^{q+1}$ has an extensible $2^{q-1}$-reading frame.
\end{lem}
\begin{proof}The output word of the algorithm is at least as long as the input word.
The while loop of the algorithm will only end if it reaches a word with pattern $\pat(w)=\alpha\overline{\alpha}$ or $\pat(w)=\alpha\overline{\alpha}\alpha$. Hence the output word will have length $2\cdot2^i$ or $3\cdot2^i$ 
for $i\in \mathbb{N}_0$. But $3\cdot 2^i \geq 2^q$ implies $i\geq q-1$.
\end{proof}

Equipped  with the algorithm, we are ready to prove the main theorem of this section. As usual we define $a \Mod b:=a-\lfloor \frac{a}{b} \rfloor b$.
\begin{theorem}\label{rf} Any factor $w$ of the infinite Thue--Morse with a given length $n:=|w|=2^q+r$, where $r<2^q$ uniquely determines at 
least $u_{\min}(n)$ and at most $u_{\max}(n)$ letters where
\[u_{\min}(n):=\begin{cases}
       0 \text{ for } n=1 \\
       0 \text{ for } n=2 \\
       0\text{ for } n=3 \\
       -n \Mod 2^{q-1}\text{ for } n>3
        \end{cases} \hspace{-2.7em}\text{and }  u_{\max}(n):=\begin{cases}
       0 \text{ for } n=1 \\
       2 \text{ for } n=2 \\
       1 \text{ for } n=3 \\
       2^{\lfloor \log_2 (n-2) \rfloor+2} -n \text{ for } n>3 .
        \end{cases}        
\]
This bounds are sharp.
\end{theorem}
\begin{proof}
Table~\ref{lookup} allows us to check the cases with $n \leq 3$. Then we take a look at the function $u_{\min}(n)$.  According to Lemma~\ref{wl} the word $w$ has an extensible
 $2^{q-1}$-reading frame and therefore determines  at least $ 2^{q-1}\cdot i-n$ letters for some $i\in \mathbb{N}_0$. But the smallest positive value of
 $ 2^{q-1}\cdot i-n$ is exactly $-n \Mod 2^{q-1}$. To show
 that it is actually possible to obtain this value  for $r\leq 2^{q-1}$ and $r> 2^{q-1}$, take the first $n$ letters of $f_{2^{q-1}}\overline{f_{2^{q-1}}}f_{2^{q-1}}$ 
 and $f_{2^q}\overline{f_{2^q}}$, respectively.
 
 To analyze the function $u_{\max}(n)$ we insert the word $w\in T^n$ in a $2^{q-1}$-reading frame, which exists according to Lemma~\ref{wl}. A word of length $2^q$ or $2^q+1$ can 
 determine $3$ frame words in the extensible $2^{q-1}$-reading frame. So after $q-1$ iterations of the while loop we have a word of length $3$. We enter the while loop again, extend the word (in the best case) to length $4$ and then
 map it via $m^{-1}$ to a word with pattern $\alpha\overline{\alpha}$. So after $q$ iterations we determined $2\cdot2^q$ letters.
 
 If the word $w$ has length $2^q+2 \leq |w|\leq 2^q+2^{q-1}+1$ it can determine $4$ frame words. So after $q-1$ iterations we have a word of length $4$ which (in the best case)
 has a $4$-reading frame and gives therefore $2$ further iterations before we end up in a word with pattern $\alpha\overline{\alpha}$. Here we determined $2\cdot2^{q+1}$ letters.
 
 If $2^q+2^{q-1}+2 \leq |w|\leq 2^{q+1}-1$ the word $w$ can determine $5$ frame words, so we have a word of length $5$ after $q-1$ iterations, extend it to length $6$, map it 
 to length $3$ and (in the best case) extend it to length $4$, before it is mapped to a word with pattern $\alpha\overline{\alpha}$. Again we determined $2\cdot2^{q+1}$ letters.
 
 In each of these cases we determined $2^{\lfloor \log_2 (n-2) \rfloor+2} -n$ new letters, but we always assumed a best case. What is left  is to show that there is always a
 word $w$, with $|w|=2^q+r$, $r<2^q$, so that the best case occurs. The first $n$ letters of $\suf_1(f_{2^{q-1}})f_{2^{q-1}}\overline{f_{2^{q-1}}}\overline{f_{2^{q-1}}}f_{2^{q-1}}$ form 
 such a word.
\end{proof}

Let us look at the relative length $\frac{|w'|}{|w|}$ of an extension, where $w$ is the input and $w'$ is the output of the algorithm. 

\begin{lem}
Let $w$ be a factor of the Thue--Morse word and let $w'$ be its unique extension.
 The relative length  taken over all $w\in T$ satisfies  $\inf  \frac{|w'|}{|w|}=1$ and $\sup  \frac{|w'|}{|w|}=4$.
\end{lem}

\begin{proof}
 We have $\frac{|w'|}{|w|}=1$ for the words $v=f_{2^q}\overline{f_{2^q}}$ of length $2\cdot 2^q$ and $w=f_{2^{q}}\overline{f_{2^{q}}}f_{2^{q}}$ of length $3\cdot 2^q$.

For the upper limit we look at the  words $w_q=\suf_1(f_{2^{q-1}})f_{2^{q-1}}\overline{f_{2^{q-1}}}1.$  These words of length $2^q+2$ have an unique extension $w'_q$ of length
$4\cdot 2^q$. Therefore we have
$\lim_{q\rightarrow \infty}\frac{|w_q'|}{|w_q|}=\lim_{q\rightarrow \infty}\frac{4\cdot 2^q}{2^q+2}=4.$
\end{proof}

\begin{ex}
The word $w=\hb{0110}\hb{1001}$ is a word of length $8$ without an extension ($\frac{|w'|}{|w|}=1$), while $v=\lhb{0}\rhb{01101}$ is a word of length 6 which can be extended to the word
$v'=\hb{\deemph{1001011}0}\hb{01101\deemph{001}}$ of length 16 ($\frac{|v'|}{|v|}=\nicefrac{8}{3}$). 
\end{ex}

\section{ The odd frame and the short coding}\label{obf}

In this section we compress the information in $\class(w)$ from Definition \ref{con}. From now on we call the extensible $2$-reading frame also \emph{even frame}. We can get
an other reading frame if we shift the even frame one letter. This
new reading fame is called \emph{odd frame}.

We use these names since a given factor occurs an infinite number of times in the Thue--Morse word but always with the same parity (for $n>3$) of the first letter. So a word is
in the \emph{even frame} if its  first letter in the Thue--Morse word has even parity and it is in the \emph{odd frame} otherwise.

\begin{ex}
 A word $01011$ can be read in the even frame $\hb{01}\hb{01}\rhb{1}$ or in the odd frame $\lhb{0}\hb{10}\hb{11}$.
\end{ex}

While the only two complete frame words in the even frame are $01$ and $10$ we have the four complete frame words $00,01,10$ and $11$ in the odd frame. We call the
odd frame words
$00$ and $11$ \emph{pairs}. The easiest way to find the odd frame of a word is to look for pairs, since pairs can only occur in the odd frame.  

We define a
\emph{short coding} for odd frame words as: 
$\hb{01},\hb{10}\mapsto \mathbf{D}$(ifferent), 
$\hb{00},\hb{11}\mapsto \mathbf{E}$(qual) and finally $\lhb{0},\lhb{1},\rhb{0},\rhb{1}\mapsto \mathbf{S}$(hort).

An odd word without a prefix and suffix in the odd frame is called \emph{ pure odd word}. The study of pure odd words will turn out to be crucial for the rest of the paper.
There is no $\mathbf{S}$ in the short coding of a pure odd word and all pure odd words have even length.

\begin{ex}
 The word $v=1100$ is a pure odd word since it has the odd frame $v=\hb{11}\hb{00}$ and the short coding $\mathbf{E}\mathbf{E}$. On the other hand the word
 $w=01001$ is not a pure odd wort since the odd frame $w=\hb{01}\hb{00}\rhb{1}$ has a single letter suffix and therefore the short coding $\mathbf{DES}$.
\end{ex}

If an odd frame word ends with one letter, the next one starts with another letter since $\lhb{\alpha}\rhb{\alpha}$ in the odd frame would be $\hb{\alpha\alpha}$ in the even frame 
which can not occur. This fact allows us to recover a word $w$ from its short coding if we know a single letter of $w$, and to recover $\pat(w)$ from the short
coding too.

\begin{ex}
Take the word $w=1001011$. It contains two pairs $1\hb{00}10\hb{11}$ and has therefore the odd frame $w=\lhb{1}\hb{00}\hb{10}\hb{11}$. 
The short coding $\mathbf{SEDE}$ gives the pattern $\lhb{\alpha}\hb{\overline{\alpha}\overline{\alpha}}\hb{\alpha\overline{\alpha}}\hb{\alpha\alpha}$ since $t$ is cubefree.  If we know
 $w_0$ we can recover $w$ from its short coding.
 \end{ex}

 Thus we can switch between patterns in the odd frame and the short coding. We will use this in the following proofs.
 
\begin{lem}\label{prop}
 The odd frame of the infinite Thue--Morse word has following properties:
 \begin{itemize}
  \item The sequence $\mathbf{DD}$ can not occur.  
  \item The sequence $\mathbf{DEED}$ can not occur.  
  \item The sequence $\mathbf{EEEE}$ can not occur.
 \end{itemize}
\end{lem}

\begin{proof}
  \begin{itemize}
  \item  We showed in Example~\ref{frame} that the word $0101$ and therefore the pattern $\alpha\overline{\alpha}\alpha\overline{\alpha}$ is in the even frame.
   \item  A word with pattern $\alpha\overline{\alpha}\alpha\alpha\overline{\alpha}\overline{\alpha}\alpha\overline{\alpha}$ has no $4$-reading frame and is therefore not in $T$.
    \item A word with pattern $\alpha\alpha\overline{\alpha}\overline{\alpha}\alpha\alpha\overline{\alpha}\overline{\alpha}$ has no $8$-reading frame and is therefore not in $T$.
 \end{itemize}
\end{proof}

 So at least every second letter in a short coding is an $\mathbf{E}$ and at
most $\nicefrac[]{3}{4}$ of the letters are $\mathbf{E}$. This gives an upper bound for the growth of the $2$-abelian complexity $\mathcal{P}_n$ of $t$.
 As consequence of Lemma~\ref{prop} two consecutive $\mathbf{E}$ have either the form $\mathbf{EE}$ or $\mathbf{EDE}$ which corresponds to the
 patterns $\alpha\alpha\overline{\alpha}\overline{\alpha}$ 
 and $\alpha\alpha\overline{\alpha}\alpha\overline{\alpha}\overline{\alpha}$. So we just proved the next Lemma.
\begin{lem} \label{alt}
 The pairs $00$ and $11$ alternate in the odd frame.
\end{lem}

Now we will use the odd frame to compress the information from Definition~\ref{con}.
To achieve this we define a function $p\colon T\rightarrow \mathbb{N}_0,\enskip w \mapsto |w|_{00}+|w|_{11}$, which counts the pairs in a word $w$, and a function
\[r(w):=\begin{cases}
   0, \text{ if $w_0w_1$ is in the odd frame};\\
   1, \text{ if  $w_0w_1$ is in in the even frame};
  \end{cases}\]
which determines the reading frame of $w$. This means a word $w$ is a pure odd word if $r(w)=0$ and $|w|$ is even. It would be possible to use the short coding to define $r(w)$: $r(w)=1$ if
the short coding of $w$ starts with $\mathbf{S}$ and $r(w)=0$ otherwise.

With the two functions $p(w)$ and $r(w)$ we can collect all information necessary to determine the $2$-abelian equivalence class of a word $w$ in a $3$-tuple:
\[ \vect\colon  T \rightarrow \mathbb{N}_0 \times \{0,1\}^2,\enskip w \mapsto (w_0, p(w), r(w)).\]

       \begin{ex}
 Let us look at Example \ref{rr} again. For the word $w=10011010$ we have now $\vect(w)=(1,2,1).$
\end{ex}

In the next theorem we show that we have all information of $\class(w)$ in  $\vect(w)$, we can recover $\class(w)$  from $\vect(w)$. We will use the \emph{XOR} operator $\oplus$ and the Iverson bracket $\llbracket S \rrbracket$
 which is $1$ if the statement $S$ is true and $0$ otherwise.
\begin{lem} \label{vec} Let $w$ be a word of length $n$.
There is a function $h$ so that 
\[h(\vect(w))=\class(w).\]
For two words $v,w \in T^n$ with $v\not =w$, we have $h(\vect(v))=h(\vect(w))$ if and only if $v_0=w_0$, $p(v)=p(w)$ and $p(v)$ is even.
\end{lem}
\begin{proof}
 The basic idea is to use parity arguments. If we take $w$ and erase one letter from every pair, we get a sequence of length $|w|-p(w)$ which alternates between $0$ and $1$ and 
 starts with $w_0$. Since the sequence has the same number of $01$ and $10$ as $w$, we can use it to determine $|w|_{01}$, $|w|_{10}$ and the last letter $w_ n$. 
 
The pairs  already form an alternating sequence (Lemma~\ref{alt}), so we only need to identify the first pair. In the odd frame a word can start either with
 $\mathbf{E}$ which corresponds to the
 pattern $\alpha\alpha$ or with
 $\mathbf{DE}$ which corresponds to the
 pattern $\alpha\overline{\alpha}\alpha\alpha$. In both cases the first pair is $w_0w_0$.   

In the even frame a word starts  with
$\mathbf{SE}$ which corresponds to the
 pattern $\alpha\overline{\alpha}\overline{\alpha}$ or with $\mathbf{SDE}$ which corresponds to the
 pattern $\alpha\overline{\alpha}\alpha\overline{\alpha}\overline{\alpha}$. In both cases the first pair is $\overline{w_0}\overline{w_0}$. This allows us
to  determine $|w|_{00}$, $|w|_{11}$. We can also give these values in an explicit form as
\[h\colon  (w_0, p(w), r(w)) \mapsto (a, b, c, d, w_0, e) \]
with 

\[a=\lfloor \tfrac{p(w)}{2} \rfloor+\llbracket p(w) \text{ odd} \rrbracket\llbracket w_0=0 \rrbracket\]
\[b=\lfloor \tfrac{|w|-p(w)}{2} \rfloor+\llbracket |w|-p(w)\text{ even}\rrbracket\llbracket w_0=0 \rrbracket\]
\[c=\lfloor  \tfrac{|w|-p(w)}{2}\rfloor+\llbracket |w|-p(w)\text{ even} \rrbracket\llbracket w_0=1 \rrbracket\]
\[d=\lfloor \tfrac{p(w)}{2} \rfloor+\llbracket  p(w) \text{ odd} \rrbracket\llbracket w_0=1 \rrbracket\]
\[e= \llbracket |w|-p(w)\text{ even}\rrbracket \oplus w_0.\]

Let $v,w \in T$ with $v\not =w$ now be two words that belong to the same equivalence class. Then $v_0=w_0$ and $p(v)=p(w)$, so they can only differ in the reading frame
with $r(v)\not =r(w)$. The reading frame determines the first pair in the alternating  pair sequence. If $p(w)$ is even, the numbers $|w|_{00}$ and $|w|_{11}$ do not
depend on $r(w)$. So $h(\vect(v))=h(\vect(w))$ for $v\not =w$ if and only if $v_0=w_0$, $p(v)=p(w)$ and $p(w)$ is even.
\end{proof}

The idea of Lemma~\ref{vec} is to gather more information in less memory. We need two boolean and four  integer variables for $\class(w)$ while $\vect(w)$ uses only
one  integer and two boolean variables.

\begin{ex}
We have $\class(w)=\class(v)$  for the  two words $w=011001$ and $v=001011$ but $\vect(w)\not=\vect(v)$ since $r(w)=1$ and $r(v)=0$. So $\vect$ can distinguish more words than
$\class$.
\end{ex}

With Lemma~\ref{vec} we can determine the $2$-abelian complexity $\mathcal{P}_n$ of $t$ if we know the possible values of $\vect(w)$, for $ w\in T_n$. Which are the possible values of $\vect(w)$?

The boolean variables can be $0$ or $1$ since
a word $w\in T_n$ can start either with $0$ or $1$
and can be in the even frame or in the odd frame. The difficult part is to find the  possible values of $p(w)$. In the next section we will find a method to  obtain them. 

\section{On pairs}\label{pa}

We are interested in how many pairs can occur in a word in $T_n$. So we define  $\mathbf{pairs}(n):=\{p(w)\mid w\in T_n\}$. It will  emerge that we will need a second set $\mathbf{PAIRS}(2n):=\{p(w)\mid r(w)=0,\,w\in T_{2n}\}$. The value
 $\mathbf{PAIRS}(n)$ is undefined for odd $n$. The elements of  $\mathbf{PAIRS}(2n)$ are the possible numbers of $\mathbf{E}$ in pure odd words of length $2n$.

\begin{ex}
Let us determine $\mathbf{pairs}(6)$ and $\mathbf{PAIRS}(6)$.

\begin{center}
 \begin{tabular}{ccc}
\toprule
  Pattern & Coding & $p(w)$  \\
  \midrule
  $\alpha\alpha\overline{\alpha}\alpha\overline{\alpha}\overline{\alpha}$ & $\mathbf{EDE}$ & $2$ \\
    $\alpha\alpha\overline{\alpha}\overline{\alpha}\alpha\alpha$ & $\mathbf{EEE}$ & $3$   \\
        $\alpha\alpha\overline{\alpha}\overline{\alpha}\alpha\overline{\alpha}$ &$\mathbf{EED}$ & $2$  \\
           $\alpha\overline{\alpha}\alpha\alpha\overline{\alpha}\alpha$ & $\mathbf{DED}$ & $1$   \\
 \bottomrule
 \end{tabular}
\hspace{2em}
 \begin{tabular}{ccc}
\toprule
  Pattern & Coding & $p(w)$\\
  \midrule
  $\alpha\overline{\alpha}\alpha\alpha\overline{\alpha}\overline{\alpha}$ & $\mathbf{DEE}$ & $2$ \\
  $\alpha\overline{\alpha}\alpha\overline{\alpha}\overline{\alpha}\alpha$ & $\mathbf{SDES}$ & $1$ \\
  $\alpha\overline{\alpha}\overline{\alpha}\alpha\alpha\overline{\alpha}$ & $\mathbf{SEES}$ & $2$  \\
  $\alpha\overline{\alpha}\overline{\alpha}\alpha\overline{\alpha}\alpha$ & $\mathbf{SEDS}$ & $1$  \\
 \bottomrule
 \end{tabular}
 \end{center}
So $\mathbf{pairs}(6)=\mathbf{PAIRS}(6)=\{1,2,3\}$.
\end{ex}

It is not always the case that $\mathbf{pairs}(2n)=\mathbf{PAIRS}(2n)$. For example $\mathbf{pairs}(8)=\{1,2,3\}$ while $\mathbf{PAIRS}(8)=\{2,3\}$.

The following theorem  is the main tool to prove results about the $2$-abelian complexity  $\mathcal{P}_n$, since all properties of 
$\mathcal{P}_n$ can be obtained from the properties
of $\mathbf{pairs}(n)$ and $\mathbf{PAIRS}(n)$.
\begin{theorem} \label{pairs}
For $n\geq 4$ the sets $\mathbf{pairs}(n)$ and $\mathbf{PAIRS}(n)$ fulfill the recursions
\begin{align}
 \mathbf{PAIRS}(2n)&=n-\mathbf{pairs}(n+1)\label{pairs1}\\
  \mathbf{pairs}(2n+1)&=\mathbf{PAIRS}(2n)\label{pairs2}\\
   \mathbf{pairs}(2n)&=\mathbf{PAIRS}(2n)\cup \mathbf{PAIRS}(2n-2)\label{pairs3}
   \end{align}
   with $n-\mathbf{pairs}(n+1):=\{n-y\mid y \in \mathbf{pairs}(n+1)\}$.
\end{theorem}
\begin{proof}
 The proof works only for $n\geq 4$ since we distinguish between even and odd frame and words of length $n < 4$ may not have a   defined $2$-reading frame.
 
 Let $w\in T_{n+1}$.  We have $|w|_{01}+|w|_{10}=n-p(w)$ since $w$ has $n$ factors of length $2$. The image $m(w)$ has length $2n+2$ and is in the even frame. 
 We remove the prefix and the suffix of  $m(w)$ to get a pure odd word $w'$ of length $2n$. We have a pair $00$ in $w'$ if and only if there is a $10$ in $w$ and a pair $11$ in $w'$ if and only if there is a $01$ in $w$.
 Thus $p(w')=n-p(w)$. Since the steps to get from $w$ to $w'$ are bijective, the two sets $\mathbf{PAIRS}(2n)$ and $n-\mathbf{pairs}(n+1)$ are equal. 
 
 All words $w\in T_{2n+1}$ are of the form $w'\mathbf{S}$ or $\mathbf{S}w'$ where $w'$ is a pure odd word of length $2n$. Since $p(w')=p(w'\mathbf{S})=p(\mathbf{S}w')$ 
 the bijection $w'\mapsto w'\mathbf{S}$ proves $\mathbf{pairs}(2n+1)=\mathbf{PAIRS}(2n)$.
 
 Every word in $w\in T_{2n}$ is either of the form $w'$ or $\mathbf{S}w''\mathbf{S}$, where $w'$ is a pure odd word of length $2n$ and $w''$ is a pure
 odd word of length $2n-2$. Again, adding and removing $\mathbf{S}$ is a bijection which does not change the number of pairs. Therefore
 $\mathbf{pairs}(2n)=\mathbf{PAIRS}(2n)\cup \mathbf{PAIRS}(2n-2)$.
\end{proof}

We write $\lbrack a, b\rbrack$ to denote the interval of all integers between $a$ and $b$, including both. The integer interval $\lbrack a, b\rbrack$ has cardinality
$\#\lbrack a, b\rbrack=b-a+1$.
For a set $S$ we define $\#_2 S$, the number of even elements in $S$, as \[\#_2 S:=\#\{s\in S\mid s\equiv 0 \pmod 2 \}.\]

\begin{lem}\label{mvs}
 The sets $\mathbf{pairs}(n)$ and $\mathbf{PAIRS}(2n)$ are integer intervals.
\end{lem}
\begin{proof}
This is true for $n<10$ (cf. Table~\ref{small}). All other values can be calculated using Theorem~\ref{pairs}. In cases (1) and (2)  of Theorem~\ref{pairs} it is obvious that integer intervals are mapped to
integer intervals, we just
have to show that $\mathbf{PAIRS}(2n)\cup \mathbf{PAIRS}(2n-2)$ is an integer interval. This is true, since we know from the definition of $\mathbf{PAIRS}(2n)$ that the upper and lower 
limit of two consecutive sets can differ only by $1$.
\end{proof}

\begin{table}[ht]
\begin{center}
 \setlength{\tabcolsep}{2pt}
\begin{tabular}[h]{cccc}
 \toprule
$n$ & $\mathbf{PAIRS}(n)$ &$\mathbf{pairs}(n)$&$\mathcal{P}_n$\\
\midrule
0&\{0\}&\{0\}&1\\
1&&\{0\}&2\\
2&\{0,1\}&\{0,1\}&4\\
3&&\{0,1\}&6\\
4&\{1,2\}&\{0,1,2\}&8\\
\bottomrule
\end{tabular}
\hspace{2em}
\begin{tabular}[h]{cccc}
 \toprule
$n$ & $\mathbf{PAIRS}(n)$ &$\mathbf{pairs}(n)$&$\mathcal{P}_n$\\
\midrule
5&&\{1,2\}&6\\
6&\{1,2,3\}&\{1,2,3\}&8\\
7&&\{1,2,3\}&10\\
8&\{2,3\}&\{1,2,3\}&8\\
9&&\{2,3\}&6\\
\bottomrule
\end{tabular}
\end{center}
\caption{Small values of the functions}\label{small}
\end{table}

In the next theorem we determine the number of $2$-abelian equivalence classes.

\begin{lem} \label{conclass}
  For $n\geq 4$ the number of $2$-abelian equivalence classes $ \mathcal{P}_n$ is given by the two formulas
 \begin{flalign*}
 \mathcal{P}_{2n+1}&=2(2\#\mathbf{PAIRS}(2n)-\#_2\mathbf{PAIRS}(2n)) \\
 \mathcal{P}_{2n}&=2(\#\mathbf{PAIRS}(2n)+\#\mathbf{PAIRS}(2n-2)\\ &\hspace{96pt} -\#_2(\mathbf{PAIRS}(2n)\cap\mathbf{PAIRS}(2n-2))).
\end{flalign*}
\end{lem}
\begin{proof}
 This is an immediate consequence of Lemma~\ref{vec} and Theorem~\ref{pairs}. First we look at $\mathcal{P}_{2n}$. We want to find all possible values of $\vect(w)$ for $w\in T_{2n}$.
 We have $\#\mathbf{PAIRS}(2n)$ possibilities to choose $p(w)$ in the odd frame and we have $\#\mathbf{PAIRS}(2n-2)$ possibilities to choose $p(w)$ in the even frame. If there
 is an even value $p(w)$ in both frames, the odd frame and the even frame give the same equivalence class so we subtract $\#_2(\mathbf{PAIRS}(2n)\cap\mathbf{PAIRS}(2n-2))$.
 Finally we multiply by $2$ since $w_0$ can be $0$ or $1$.
 
 We use the same argument for $\mathcal{P}_{2n+1}$ but we have $\mathbf{PAIRS}(2n)$  possible pairs  in both frames and therefore also the intersection.
\end{proof}

Now we look at several sets, because for even numbers the recursion needs two sets. If we know $\mathbf{pairs}(n)$ and $\mathbf{pairs}(n+1)$, we can use 
Theorem~\ref{pairs} to determine
$\mathbf{PAIRS}(2n-2)$ and $\mathbf{PAIRS}(2n)$ and thus $\mathbf{pairs}(2n-1)$, $\mathbf{pairs}(2n)$ and $\mathbf{pairs}(2n+1)$. With Lemma~\ref{conclass} we can then
determine $ \mathcal{P}_{2n-1}$, $\mathcal{P}_{2n}$ and $ \mathcal{P}_{2n+1}$.

It is clear from the definition of $\mathbf{pairs}(n)$ that the sequence $\min\mathbf{pairs}(n)$ is monotonically increasing in steps of $0$ or $1$:
\[ \min\mathbf{pairs}(n+1)-\min\mathbf{pairs}(n) \in\{0,1\}.\]
This  also holds for $\max\mathbf{pairs}(n)$, $\min\mathbf{PAIRS}(n)$ and $\max\mathbf{PAIRS}(n)$. So if $\mathbf{pairs}(n)=\lbrack a,b\rbrack$ we have four possibilities for
$\mathbf{pairs}(n+1)$:
\[\lbrack a,b\rbrack,\enskip \lbrack a+1,b\rbrack,\enskip\lbrack a,b+1\rbrack\text{ and }\lbrack a+1,b+1\rbrack.\]

Now we use Theorem~\ref{pairs} to make Table~\ref{MVS}. 

If we observe small values of the sequence  $\mathbf{pairs}(n)$ we see that the $IV.$ case does not occur.
The column   $\mathbf{PAIRS}(2n{+}i)$ of Table~\ref{MVS} shows that case $IV.$ can also not occur as image of smaller values. Therefore the case $IV.$ can not occur anywhere and
we have just proved:
\begin{equation}\label{eq1}
\mathbf{PAIRS}(2n-2)\not= \mathbf{PAIRS}(2n)\text{ and }\mathbf{pairs}(n+1)\not=1+\mathbf{pairs}(n).
\end{equation}

\begin{table}[h!bt]
 \setlength{\tabcolsep}{2pt}
\begin{footnotesize}
\begin{center}
\begin{tabular}[h]{cccllll}
\toprule
Case & $i$ & $\mathbf{pairs}(n{+}i)$ &  $\mathbf{PAIRS}(2n{+}i)$ & $\mathbf{pairs}(2n{+}i)$&   $\mathbf{PAIRS}(4n{+}i)$ & $\mathbf{pairs}(4n{+}i)$\\
\midrule
\multirow{4}{*}{I.} & $\scriptstyle{-}2$ & &\hspace{-0.6em}$\lbrack n{-}b{-}1,n{-}a{-}1\rbrack $&  &$\lbrack n{+}a{-}1,n{+}b\rbrack $&\\
  & $\scriptstyle{-}1$  & &&$\lbrack n{-}b{-}1,n{-}a{-}1\rbrack $&&\hspace{-0.3em}$\lbrack n{+}a{-}1,n{+}b\rbrack $\\
    & $\scriptstyle0$&$\lbrack a,b\rbrack $ &\hspace{-0.6em}$\lbrack n{-}b,n{-}a\rbrack $&$\lbrack n{-}b{-}1,n{-}a\rbrack $&$\lbrack n{+}a,n{+}b\rbrack $&\hspace{-0.3em}$\lbrack n{+}a{-}1,n{+}b\rbrack $\\
    &  $\scriptstyle{+}1$&$\lbrack a,b\rbrack $ & &$\lbrack n{-}b,n{-}a\rbrack $&&\hspace{-0.3em}$\lbrack n{+}a,n{+}b\rbrack $ \\
 \midrule
    \multirow{4}{*}{II.} & $\scriptstyle{-}2$ &  &\hspace{-0.6em}$\lbrack n{-}b{-}1,n{-}a{-}1\rbrack $&&$\lbrack n{+}a,n{+}b\rbrack $& \\
  & $\scriptstyle{-}1$  & &&$\lbrack n{-}b{-}1,n{-}a{-}1\rbrack $&&\hspace{-0.3em}$\lbrack n{+}a,n{+}b\rbrack $\\
    & $\scriptstyle0$&$\lbrack a,b\rbrack $ &\hspace{-0.6em}$\lbrack n{-}b,n{-}a{-}1\rbrack $&$\lbrack n{-}b{-}1,n{-}a{-}1\rbrack $&$\lbrack n{+}a{+}1,n{+}b\rbrack $&\hspace{-0.3em}$\lbrack n{+}a,n{+}b\rbrack $\\
    &  $\scriptstyle{+}1$&$\lbrack a{+}1,b\rbrack $ &&$\lbrack n{-}b,n{-}a{-}1\rbrack $ & &\hspace{-0.3em}$\lbrack n{+}a{+}1,n{+}b\rbrack $\\
 \midrule
    \multirow{4}{*}{III.} & $\scriptstyle{-}2$ &  &\hspace{-0.6em}$\lbrack n{-}b{-}1,n{-}a{-}1\rbrack $& &$\lbrack n{+}a{-}1,n{+}b\rbrack $& \\
  & $\scriptstyle{-}1$  & &&$\lbrack n{-}b{-}1,n{-}a{-}1\rbrack $&&\hspace{-0.3em}$\lbrack n{+}a{-}1,n{+}b\rbrack $\\
    & $\scriptstyle0$&$\lbrack a,b\rbrack $ &\hspace{-0.6em}$\lbrack n{-}b{-}1,n{-}a\rbrack $&$\lbrack n{-}b{-}1,n{-}a\rbrack $&$\lbrack n{+}a,n{+}b{+}1\rbrack $&\hspace{-1.6em}$\lbrack n{+}a{-}1,n{+}b{+}1\rbrack $\\
    &  $\scriptstyle{+}1$&$\lbrack a,b{+}1\rbrack $ & &$\lbrack n{-}b{-}1,n{-}a\rbrack $&&\hspace{-0.3em}$\lbrack n{+}a,n{+}b{+}1\rbrack $\\
\midrule
    \multirow{4}{*}{IV.} & $\scriptstyle{-}2$ &  &\hspace{-0.6em}$\lbrack n{-}b{-}1,n{-}a{-}1\rbrack $&  &$\lbrack n{+}a,n{+}b\rbrack $&\\
  & $\scriptstyle{-}1$  & &&$\lbrack n{-}b{-}1,n{-}a{-}1\rbrack $&&\hspace{-0.3em}$\lbrack n{+}a,n{+}b\rbrack $\\
    & $\scriptstyle0$&$\lbrack a,b\rbrack $ &\hspace{-0.6em}$\lbrack n{-}b{-}1,n{-}a{-}1\rbrack $&$\lbrack n{-}b{-}1,n{-}a{-}1\rbrack $&$\lbrack n{+}a{+}1,n{+}b{+}1\rbrack $&\hspace{-0.3em}$\lbrack n{+}a,n{+}b{+}1\rbrack $\\
    &  $\scriptstyle{+}1$&$\lbrack a{+}1,b{+}1\rbrack $ & &$\lbrack n{-}b{-}1,n{-}a{-}1\rbrack $ &&\hspace{-1.35em}$\!\lbrack n{+}a{+}1,n{+}b{+}1\rbrack $\\
\bottomrule
\end{tabular}
\end{center}
\end{footnotesize}
 \caption{Behavior of integer intervals under the maps $\mathbf{pairs}(n)$ and $\mathbf{PAIRS}(n)$.}\label{MVS}
\end{table}

\begin{ex}
 Let us take a look at the first case in Table~\ref{MVS}. We start with the two sets  $\mathbf{pairs}(n)=\lbrack a,b\rbrack$ and  $\mathbf{pairs}(n+1)=\lbrack a,b\rbrack$.
 Now we use Equation~(\ref{pairs1}) of Theorem~\ref{pairs} and
 get
 \[\mathbf{PAIRS}(2n)=n-\mathbf{pairs}(n+1)=n-\lbrack a,b\rbrack=\lbrack n-b,n-a\rbrack\]
 and
  \[\mathbf{PAIRS}(2n-2)=n-1-\mathbf{pairs}(n)=n-1-\lbrack a,b\rbrack=\lbrack n-b-1,n-a-1\rbrack.\]
  Then we use Theorem~\ref{pairs} Equation~(\ref{pairs2}) to get
  \[\mathbf{pairs}(2n-1)=\mathbf{PAIRS}(2n-2)=\lbrack n-b-1,n-a-1\rbrack\]
  and
   \[\mathbf{pairs}(2n+1)=\mathbf{PAIRS}(2n)=\lbrack n-b,n-a\rbrack.\]
   Finally we use   use Theorem~\ref{pairs} Equation~(\ref{pairs3}) and get
    \[\mathbf{pairs}(2n)=\mathbf{PAIRS}(2n-2)\cup\mathbf{PAIRS}(2n)=\lbrack n-b-1,n-a-1\rbrack\cup\lbrack n-b,n-a\rbrack=\]
    \[=\lbrack n-b-1,n-a\rbrack.\]
    We continue in the same manner to  complete Table~\ref{MVS}.
\end{ex}

From the way we made Table~\ref{MVS} we can also deduce the next lemma about consecutive values of $\mathbf{pairs}(n)$, $\mathbf{PAIRS}(n)$ and $\mathcal{P}_{n}$.
\begin{lem} \label{rec}
 If we know $\mathbf{pairs}(n)$ for all $n\in\lbrack a,b\rbrack $ we can determine $\mathbf{pairs}(n')$ and $\mathcal{P}_{n'}$ for $n'\in\lbrack 2a-1,2b-1\rbrack $
 and  $\mathbf{PAIRS}(n'')$ for $n''\in\lbrack 2a-2,2b-2\rbrack $.
\end{lem}
\begin{rem}
 An especially interesting case of Lemma~\ref{rec} is to go from $\mathbf{pairs}(n)$, $n\in\lbrack 2^{q-1}+1,2^q+1\rbrack $ to $\mathbf{pairs}(n')$, $n'\in\lbrack 2^q+1;2^{q+1}+1\rbrack $. This is
 the idea behind the proof of Theorem~\ref{pal}.
\end{rem}

\section{ The sequence $\mathcal{P}_n$ is $2$-regular}\label{2reg}

We will prove Theorem~\ref{main}  by proving the $13$ relations of Theorem 1.1.* which generate all sequences for the $\mathbb{Z}$-kernel. If Theorem 1.1.* is true
the $\mathbb{Z}$-kernel is finitely generated and Theorem~\ref{main} follows.

\begin{customthm}{1.1.*}
The  $\mathbb{Z}$-kernel of the Thue-Morse word is generated by the $13$ relations:

 \begin{enumerate}
\item $\mathcal{P}_{4n+1}=\mathcal{P}_{2n+1}$
\item $\mathcal{P}_{8n+4}=\mathcal{P}_{8n+3}+\mathcal{P}_{4n+3}-\mathcal{P}_{4n+2}$
\item $\mathcal{P}_{16n}=\mathcal{P}_{8n}$
\item $\mathcal{P}_{16n+2}=\mathcal{P}_{8n+2}$
\item $\mathcal{P}_{16n+6}=-\mathcal{P}_{16n+3}+\mathcal{P}_{8n+3}+3\mathcal{P}_{8n+2}+\mathcal{P}_{4n+3}-2\mathcal{P}_{4n+2}-\mathcal{P}_{2n+1}$
\item $\mathcal{P}_{16n+7}=-\mathcal{P}_{16n+3}+\mathcal{P}_{8n+3}+3\mathcal{P}_{8n+2}+2\mathcal{P}_{4n+3}-3\mathcal{P}_{4n+2}-\mathcal{P}_{2n+1}$
\item $\mathcal{P}_{16n+8}=\mathcal{P}_{8n+2}+\mathcal{P}_{4n+3}-\mathcal{P}_{2n+1}$
\item $\mathcal{P}_{16n+10}=\mathcal{P}_{8n+2}+\mathcal{P}_{4n+3}-\mathcal{P}_{2n+1}$
\item $\mathcal{P}_{16n+11}=-\mathcal{P}_{16n+3}+3\mathcal{P}_{8n+2}+\mathcal{P}_{4n+3}-2\mathcal{P}_{2n+1}$
\item $\mathcal{P}_{16n+14}=\mathcal{P}_{16n+3}+\mathcal{P}_{8n+7}-\mathcal{P}_{8n+3}-\mathcal{P}_{8n+2}-\mathcal{P}_{4n+3}+3\mathcal{P}_{4n+2}-\mathcal{P}_{2n+1}$
\item $\mathcal{P}_{16n+15}=\mathcal{P}_{16n+3}+2\mathcal{P}_{8n+7}-3\mathcal{P}_{8n+6}-2\mathcal{P}_{8n+3}+6\mathcal{P}_{4n+2}-3\mathcal{P}_{2n+1}$
\item $\mathcal{P}_{32n+3}=\mathcal{P}_{8n+3}$
\item $\mathcal{P}_{32n+19}=-\mathcal{P}_{16n+3}+\mathcal{P}_{8n+3}+3\mathcal{P}_{8n+2}+2\mathcal{P}_{4n+3}-3\mathcal{P}_{4n+2}-\mathcal{P}_{2n+1}$
\end{enumerate}
\end{customthm}

\begin{proof}
If we observe the right hand side of these $13$ relations we see that every sequence $\mathcal{P}_{2^qn+r}$
is a linear combination of $\mathcal{P}_{2n+1}$, $\mathcal{P}_{4n+2}$, $\mathcal{P}_{4n+3}$, $\mathcal{P}_{8n}$, $\mathcal{P}_{8n+2}$, $\mathcal{P}_{8n+3}$, $\mathcal{P}_{8n+6}$,
$\mathcal{P}_{8n+7}$ and $\mathcal{P}_{16n+3}$.

To see that these $13$ relations generate all sequences for the $\mathbb{Z}$-kernel we just have to check that the left hand side of this $13$ relations cover all 
residue classes
modulo $32$. Please note that we have to write all relations modulo 32, so e.g. instead of $\mathcal{P}_{16n+7}$ we write $\mathcal{P}_{16(2n)+7}=\mathcal{P}_{32n+7}$ and
$\mathcal{P}_{16(2n+1)+7}=\mathcal{P}_{32n+23}$.
\end{proof}

 These relations have been found by computer experiments by Parreau, Rigo and Vandomme. An algorithm for finding such relations for an $\ell$-abelian sequence
 can be found in Chapter 6 of \cite{allouche2003}. We will prove the relations one by one by a  four step approach. 
 
 In this four step approach  Table~\ref{values} will be used in the second and third step\footnote{For Relation 3 we need a different table.}. It has been generated in the same manner as Table~\ref{MVS} starting with all three possibilities for the relative 
  sizes of $\mathbf{pairs}(n{+}1)$ and $\mathbf{pairs}(n{+}2)$.
  
  So let us now describe the four steps in detail, we will need them over and over again for all thirteen relations.
 
 \begin{description}
\item[First step]  We  move all terms of a relation to the right hand side and replace them there using Lemma~\ref{conclass}. 
We then  divide
  the whole equation by $2$ (as consequence of
  Lemma~\ref{conclass} all $\mathcal{P}_n$ are even for $n>0$).

  \begin{ex}[Relation 1]
   To make the steps clearer we will prove Relation~1 as example. To prove $\mathcal{P}_{4n+1}=\mathcal{P}_{2n+1}$ we move all terms to the right side.
   \[0=-\mathcal{P}_{4n+1}+\mathcal{P}_{2n+1}\]
   Now we use Lemma~\ref{conclass} to replace them and get
   \begin{small}
\[0={-}2(2\#\mathbf{PAIRS}(4n){-}\#_2\mathbf{PAIRS}(4n)){+}2(2\#\mathbf{PAIRS}(2n){-}\#_2\mathbf{PAIRS}(2n)).\]
\end{small}
Finally divide this equation by $2$ to get
\begin{small}
\[0={-}(2\#\mathbf{PAIRS}(4n){-}\#_2\mathbf{PAIRS}(4n)){+}(2\#\mathbf{PAIRS}(2n){-}\#_2\mathbf{PAIRS}(2n)).\]
\end{small}
  \end{ex}

\item[Second step]  We use Table~\ref{values} to  substitute the  term $\#\mathbf{PAIRS}(.)$ by cardinalities of suitably chosen integer intervals
(we keep $\#_2\mathbf{PAIRS}(.)$).

  Since there are three cases we will now
  have three equations.  We  calculate the cardinalities on the right hand side in all three cases, simplify until we have a single integer 
  and put the result into a triple $(\rhs_1,\rhs_2,\rhs_3)$,
  where $\rhs_i$ is the integer on the right hand side in the $i$-th case.
  
\begin{ex}[Relation 1, continued]
We start with
\begin{small}
 \[0={-}(2\#\mathbf{PAIRS}(4n){-}\#_2\mathbf{PAIRS}(4n)){+}(2\#\mathbf{PAIRS}(2n){-}\#_2\mathbf{PAIRS}(2n)).\] 
\end{small}
We look at case I. of Table~\ref{values} and substitute  $\#\mathbf{PAIRS}(2n)=\#\lbrack n{-}b,n{-}a\rbrack $ and  $\#\mathbf{PAIRS}(4n)=\#\lbrack n{+}a,n{+}b\rbrack $ to get
 \[0={-}2\#\lbrack n{+}a,n{+}b\rbrack {+}\#_2\mathbf{PAIRS}(4n){+}2\#\lbrack n{-}b,n{-}a\rbrack {-}\#_2\mathbf{PAIRS}(2n).\]
 Now we calculate the cardinalities. Since $\#\lbrack a,b\rbrack =b-a+1$ we get
 \[0={-}2(b-a+1){+}\#_2\mathbf{PAIRS}(4n){+}2(b-a+1){-}\#_2\mathbf{PAIRS}(2n).\]
 This can be simplified to \[0=\#_2\mathbf{PAIRS}(4n){-}\#_2\mathbf{PAIRS}(2n)\] so the integer in case I. is $0$. The other two cases are identical and we get
 \[0=\#_2\mathbf{PAIRS}(4n){-}\#_2\mathbf{PAIRS}(2n)+(0,0,0).\]
 This is a shorthand for three (identical) equations. Most times two or even three cases will be identical.
\end{ex}

\item[Third step] We will use Table~\ref{values} again to substitute the  cardinalities $\#_2\mathbf{PAIRS}(.)$ in all three cases and intersect them if necessary.

\begin{ex}[Relation 1, continued]
In all three cases we get
 \[0=\#_2\lbrack n{+}a,n{+}b\rbrack {-}\#_2\lbrack n{-}b,n{-}a\rbrack .\]
\end{ex}

\item[Fourth step]  Now we deal with the number of even elements in integer intervals. If two intervals of the same size with different signs contain the same
  number of even elements they cancel. 
  If the do not cancel  (because their borders have the wrong parity) or have different sizes we  split off Iverson brackets
from the beginning or end of the larger set until both sets have the same parity of limits and the same cardinality
\[\#_2\lbrack a,b\rbrack =\llbracket a \text{ even} \rrbracket+\#_2\lbrack a{+}1,b\rbrack ; \quad \#_2\lbrack a,b\rbrack =\#_2\lbrack a,b{-}1\rbrack +\llbracket b \text{ even} \rrbracket.\]

To check two intervals of the same size with different signs contain the same
  number of even elements we 
 use the following procedure, which we call \emph{normalization}:
\begin{itemize}
    \item Replace all even numbers by $0$ and all odd numbers by $1$.
    \item Change all ``$-$'' signs to ``$+$'' signs.
    \item If an $a$ occurs in the upper limit of a set we change the upper and lower limit.
   \end{itemize}
   
Its easy to see that parity of  the interval borders does not chance during the procedure, basically we calculate modulo $2$. It does however change the size of the 
intervals so it is important to use normalization only on intervals of the same size.

 \begin{ex}
   We want to show that 
   \[0=-\#_2\lbrack 5n+a+1,5n+b+2\rbrack +\#_2\lbrack 3n-b,3n-a+1\rbrack .\]
   Both sides have the same  cardinality $b-a+2$, therefore we can normalize them:
   \begin{align*}
      -\#_2\lbrack 5n{+}a{+}1,5n{+}b{+}2\rbrack +\#_2\lbrack 3n{-}b,3n{-}a{+}1\rbrack &=\\
   =-\#_2\lbrack n{+}a{+}1,n{+}b\rbrack +\#_2\lbrack n{-}b,n{-}a{+}1\rbrack &=\\
   =-\#_2\lbrack n{+}a{+}1,n{+}b\rbrack +\#_2\lbrack n{+}b,n{+}a{+}1\rbrack &=\\
    =-\#_2\lbrack n{+}a{+}1,n{+}b\rbrack +\#_2\lbrack n{+}a{+}1,n{+}b\rbrack &=0.
   \end{align*}
  \end{ex}
  
  \begin{ex}[Relation 1, continued] In all three cases we have
   \[0=\#_2\lbrack n{+}a,n{+}b\rbrack {-}\#_2\lbrack n{-}b,n{-}a\rbrack .\]
   Since the sets have the same cardinality so we normalize and get
    \[0=\#_2\lbrack n{+}a,n{+}b\rbrack {-}\#_2\lbrack n{-}b,n{-}a\rbrack =\]
        \[=\#_2\lbrack n{+}a,n{+}b\rbrack {-}\#_2\lbrack n{+}b,n{+}a\rbrack =\]
            \[=\#_2\lbrack n{+}a,n{+}b\rbrack {-}\#_2\lbrack n{+}a,n{+}b\rbrack \]
            which is true.
    
  \end{ex}

   \end{description}

      The only non mechanical step in the whole procedure is step four. Sometimes we have to split intervals before we can normalize. This will be mentioned in the proof.
After these four steps  we will have an  equation which is trivially true.
 
 \begin{proof}[Proof of Theorem~\ref{main}.]
 \hspace{3pt}
\begin{enumerate}
\item This is shown in the example.
\item We will show the first relation in more detail. We want to prove
\[ 0=-\mathcal{P}_{8n+4}+\mathcal{P}_{8n+3}+\mathcal{P}_{4n+3}-\mathcal{P}_{4n+2}. \]
In the first step we use  Lemma~\ref{conclass} and divide by $2$ to get
\begin{footnotesize}
\begin{align*}
 0=&{-}\#\mathbf{PAIRS}(8n{+}4){-}\#\mathbf{PAIRS}(8n{+}2){+}\#_2(\mathbf{PAIRS}(8n{+}4)\cap\mathbf{PAIRS}(8n{+}2))\\
&{+}2\#\mathbf{PAIRS}(8n{+}2){-}\#_2\mathbf{PAIRS}(8n{+}2){+}2\#\mathbf{PAIRS}(4n{+}2){-}\#_2\mathbf{PAIRS}(4n{+}2)\\
&{-}\#\mathbf{PAIRS}(4n{+}2){-}\#\mathbf{PAIRS}(4n){+}\#_2(\mathbf{PAIRS}(4n{+}2)\cap\mathbf{PAIRS}(4n)).
\end{align*}
\end{footnotesize}
For the second step we substitute the intervals from Table~\ref{values}.
With the shorthand $\mathcal{C}:=b-a+1$ we get in the first case  (where $\mathbf{pairs}(n+1)=\lbrack a,b\rbrack $ and  $\mathbf{pairs}(n+2)=\lbrack a,b\rbrack $)
\begin{small}
\begin{align*}
 0=&{-}2\mathcal{C}{-}2{+}\#_2(\mathbf{PAIRS}(8n{+}4)\cap\mathbf{PAIRS}(8n{+}2)){+}2\mathcal{C}{+}2{-}\#_2\mathbf{PAIRS}(8n{+}2)\\
&{+}2\mathcal{C}{+}2{-}\#_2\mathbf{PAIRS}(4n{+}2){-}2\mathcal{C}{-}1{+}\#_2(\mathbf{PAIRS}(4n{+}2)\cap\mathbf{PAIRS}(4n)).  
\end{align*} 
\end{small}
Most terms cancel and we have $1$ remaining on the right hand side. It turns out this is also true in the two other cases so our triple is $(1,1,1)$ and we get
\begin{align*}\label{8k+4}
0=&+\#_2(\mathbf{PAIRS}(8n{+}4)\cap\mathbf{PAIRS}(8n{+}2))-\#_2\mathbf{PAIRS}(8n{+}2)\\
&-\#_2\mathbf{PAIRS}(4n{+}2)+\#_2(\mathbf{PAIRS}(4n{+}2)\cap\mathbf{PAIRS}(4n))+(1,1,1).
 \end{align*}
Now we apply the third step, intersect the sets and get in all three cases
\begin{align*}
0=&+\#_2\lbrack 3n{-}b{+}1,3n{-}a{+}1\rbrack -\#_2\lbrack 3n{-}b,3n{-}a{+}1\rbrack \\
 &-\#_2\lbrack n{+}a,n{+}b{+}1\rbrack +\#_2\lbrack n{+}a,n{+}b\rbrack +1.
\end{align*}
In the fourth step we can not normalize, so we split off Iverson brackets and  get
\begin{align*}
0=&+\#_2\lbrack 3n{-}b{+}1,3n{-}a{+}1\rbrack -\llbracket 3n-b\text{ even} \rrbracket-\#_2\lbrack 3n{-}b{+}1,3n{-}a{+}1\rbrack \\
 &-\#_2\lbrack n{+}a,n{+}b\rbrack -\llbracket n+b+1\text{ even} \rrbracket+\#_2\lbrack n{+}a,n{+}b\rbrack +1.
\end{align*}
The intervals cancel and we are left with the true equation
\begin{equation*}
0=-\llbracket 3n-b\text{ even} \rrbracket - \llbracket n+b+1\text{ even} \rrbracket+1,
\end{equation*}
in all three cases.

\item We can get the values needed if we extend Table~\ref{MVS}.

\begin{tabular}[h]{rll}
\toprule
i & $\mathbf{PAIRS}(8n{+}i)$  & $\mathbf{PAIRS}(16n{+}i)$ \\
 \midrule
 $-2$& $\lbrack 3n{-}b{-}1,3n{-}a\rbrack $ & $\lbrack 5n{+}a{-}1,5n{+}b\rbrack $ \\
 $0$& $\lbrack 3n{-}b,3n{-}a\rbrack $ & $\lbrack 5n{+}a,5n{+}b\rbrack $\\
  \midrule
   $-2$& $\lbrack 3n{-}b{-}1,3n{-}a{-}1\rbrack $ & $\lbrack 5n{+}a,5n{+}b\rbrack $ \\
 $0$& $\lbrack 3n{-}b,3n{-}a{-}1\rbrack $ & $\lbrack 5n{+}a{+}1,5n{+}b\rbrack $\\
  \midrule
   $-2$& $\lbrack 3n{-}b{-}2,3n{-}a\rbrack $ & $\lbrack 5n{+}a{-}1,5n{+}b{+}1\rbrack $ \\
 $0$& $\lbrack 3n{-}b{-}1,3n{-}a\rbrack $ & $\lbrack 5n{+}a,5n{+}b{+}1\rbrack $\\
 \bottomrule
\end{tabular}

 The triple is $(0,0,0)$. The rest is a straightforward. In the fourth step we just have to normalize. 

\item Similar to relation 3. but with values from Table~\ref{values}.  The triple is $(0,0,0)$ and  we just have to normalize. 

\begin{sidewaystable}[p]
 \setlength{\tabcolsep}{4pt}
\begin{center}
\begin{tabular}[h]{cccccccc}
\toprule
Case &$i$ & $\mathbf{pairs}(n{+}i)$ & $\mathbf{PAIRS}(2n{+}i)$ & $\mathbf{PAIRS}(4n{+}i)$& $\mathbf{PAIRS}(8n{+}i)$   &$\mathbf{PAIRS}(16n{+}i)$ & $\mathbf{PAIRS}(32n{+}i)$    \\
 \midrule
\multirow{11}{*}{I.}& 0 &      & $\lbrack n{-}b,n{-}a\rbrack $    & $\lbrack n{+}a,n{+}b\rbrack $     & $\lbrack 3n{-}b,3n{-}a\rbrack $     &$\lbrack 5n{+}a,5n{+}b\rbrack $    &$\lbrack 11n{-}b,11n{-}a\rbrack $  \\
& 1 & $\lbrack a,b\rbrack $         &   &       &   &       &      \\
 &2 & $\lbrack a,b\rbrack $  &$\lbrack n{-}b{+}1,n{-}a{+}1\rbrack $   &$\lbrack n{+}a,n{+}b{+}1\rbrack $   & $\lbrack 3n{-}b,3n{-}a{+}1\rbrack $  &$\lbrack 5n{+}a,5n{+}b{+}1\rbrack $ &$\lbrack 11n{-}b,11n{-}a{+}1\rbrack $ \\
& 4 &   &            & $\lbrack n{+}a{+}1,n{+}b{+}1\rbrack $  & $\lbrack 3n{-}b{+}1,3n{-}a{+}2\rbrack $   & $\lbrack 5n{+}a{+}1,5n{+}b{+}2\rbrack $  & $\lbrack 11n{-}b{+}1,11n{-}a{+}2\rbrack $ \\
 &6 &   &                  &   & $\lbrack 3n{-}b{+}2,3n{-}a{+}3\rbrack $  & $\lbrack 5n{+}a{+}1,5n{+}b{+}3\rbrack $   &$\lbrack 11n{-}b{+}1,11n{-}a{+}3\rbrack $\\
 &8 &   &              &   & $\lbrack 3n{-}b{+}3,3n{-}a{+}3\rbrack $  & $\lbrack 5n{+}a{+}2,5n{+}b{+}3\rbrack $  & $\lbrack 11n{-}b{+}2,11n{-}a{+}3\rbrack $\\
 &10 &   &                     &   &   &$\lbrack 5n{+}a{+}2,5n{+}b{+}4\rbrack $   & $\lbrack 11n{-}b{+}2,11n{-}a{+}4\rbrack $ \\
 &12 &   &                      &   &   & $\lbrack 5n{+}a{+}3,5n{+}b{+}4\rbrack $ &$\lbrack 11n{-}b{+}3,11n{-}a{+}5\rbrack $   \\
 &14 &   &                      &   &   &$\lbrack 5n{+}a{+}4,5n{+}b{+}5\rbrack $  &$\lbrack 11n{-}b{+}4,11n{-}a{+}6\rbrack $ \\
& 16 &   &                       &   &   & $\lbrack 5n{+}a{+}5,5n{+}b{+}5\rbrack $  & $\lbrack 11n{-}b{+}5,11n{-}a{+}6\rbrack $  \\
 &18 &   &                     &   &   &   &  $\lbrack 11n{-}b{+}5,11n{-}a{+}7\rbrack $  \\
 \midrule
\multirow{11}{*}{II.}& 0 &      &$\lbrack n{-}b,n{-}a\rbrack $     &$\lbrack n{+}a,n{+}b\rbrack $     &$\lbrack 3n{-}b,3n{-}a\rbrack $      & $\lbrack 5n{+}a,5n{+}b\rbrack $   &$\lbrack 11n{-}b,11n{-}a\rbrack $ \\
& 1 & $\lbrack a,b\rbrack $         &   &       &   &       &      \\
& 2 & $\lbrack a{+}1,b\rbrack $  & $\lbrack n{-}b{+}1,n{-}a\rbrack $  & $\lbrack n{+}a{+}1,n{+}b{+}1\rbrack $   & $\lbrack 3n{-}b,3n{-}a{+}1\rbrack $  & $\lbrack 5n{+}a,5n{+}b{+}1\rbrack $   & $\lbrack 11n{-}b,11n{-}a{+}1\rbrack $  \\
& 4 &   &            &$\lbrack n{+}a{+}2,n{+}b{+}1\rbrack $  & $\lbrack 3n{-}b{+}1,3n{-}a{+}1\rbrack $   &$\lbrack 5n{+}a{+}1,5n{+}b{+}2\rbrack $  &$\lbrack 11n{-}b{+}1,11n{-}a{+}2\rbrack $  \\
& 6 &   &             &     & $\lbrack 3n{-}b{+}2,3n{-}a{+}2\rbrack $  &$\lbrack 5n{+}a{+}2,5n{+}b{+}3\rbrack $  & $\lbrack 11n{-}b{+}1,11n{-}a{+}3\rbrack $ \\
& 8 &   &                &   & $\lbrack 3n{-}b{+}3,3n{-}a{+}2\rbrack $   &$\lbrack 5n{+}a{+}3,5n{+}b{+}3\rbrack $   &$\lbrack 11n{-}b{+}2,11n{-}a{+}3\rbrack $  \\
& 10 &   &                &        &   &$\lbrack 5n{+}a{+}3,5n{+}b{+}4\rbrack $  &$\lbrack 11n{-}b{+}2,11n{-}a{+}4\rbrack $ \\
 &12 &   &               &         &   & $\lbrack 5n{+}a{+}4,5n{+}b{+}4\rbrack $ & $\lbrack 11n{-}b{+}3,11n{-}a{+}4\rbrack $ \\
 &14 &   &                      &   &   &$\lbrack 5n{+}a{+}5,5n{+}b{+}5\rbrack $  & $\lbrack 11n{-}b{+}4,11n{-}a{+}5\rbrack $ \\
 &16 &   &                      &   &   & $\lbrack 5n{+}a{+}6,5n{+}b{+}5\rbrack $  & $\lbrack 11n{-}b{+}5,11n{-}a{+}5\rbrack $  \\
 &18 &   &                    &   &   &   &  $\lbrack 11n{-}b{+}5,11n{-}a{+}6\rbrack $  \\
 \midrule
\multirow{11}{*}{III.}& 0 &      & $\lbrack n{-}b,n{-}a\rbrack $    &$\lbrack n{+}a,n{+}b\rbrack $    & $\lbrack 3n{-}b,3n{-}a\rbrack $    & $\lbrack 5n{+}a,5n{+}b\rbrack $    &$\lbrack 11n{-}b,11n{-}a\rbrack $\\
 &1 & $\lbrack a,b\rbrack $        &   &   &       &       &    \\
& 2 & $\lbrack a,b{+}1\rbrack $  &$\lbrack n{-}b,n{-}a{+}1\rbrack $   &$\lbrack n{+}a,n{+}b{+}1\rbrack $   &$\lbrack 3n{-}b,3n{-}a{+}1\rbrack $  &$\lbrack 5n{+}a,5n{+}b{+}1\rbrack $    & $\lbrack 11n{-}b,11n{-}a{+}1\rbrack $ \\
& 4 &   &            & $\lbrack n{+}a{+}1,n{+}b{+}2\rbrack $   &$\lbrack 3n{-}b{+}1,3n{-}a{+}2\rbrack $    & $\lbrack 5n{+}a{+}1,5n{+}b{+}2\rbrack $  &$\lbrack 11n{-}b{+}1,11n{-}a{+}2\rbrack $ \\
& 6 &   &               &   & $\lbrack 3n{-}b{+}1,3n{-}a{+}3\rbrack $  & $\lbrack 5n{+}a{+}1,5n{+}b{+}3\rbrack $ &$\lbrack 11n{-}b{+}1,11n{-}a{+}3\rbrack $ \\
& 8 &   &                 &   & $\lbrack 3n{-}b{+}2,3n{-}a{+}3\rbrack $  & $\lbrack 5n{+}a{+}2,5n{+}b{+}3\rbrack $   & $\lbrack 11n{-}b{+}2,11n{-}a{+}3\rbrack $\\
&10 &   &                    &   &   & $\lbrack 5n{+}a{+}2,5n{+}b{+}4\rbrack $ &$\lbrack 11n{-}b{+}2,11n{-}a{+}4\rbrack $ \\
& 12 &   &                    &   &   & $\lbrack 5n{+}a{+}3,5n{+}b{+}5\rbrack $  & $\lbrack 11n{-}b{+}3,11n{-}a{+}5\rbrack $  \\
& 14 &   &                     &   &   &$\lbrack 5n{+}a{+}4,5n{+}b{+}6\rbrack $   &$\lbrack 11n{-}b{+}4,11n{-}a{+}6\rbrack $  \\
& 16 &   &                    &   &   & $\lbrack 5n{+}a{+}5,5n{+}b{+}6\rbrack $  & $\lbrack 11n{-}b{+}5,11n{-}a{+}6\rbrack $ \\
& 18 &   &                   &   &   &    &$\lbrack 11n{-}b{+}5,11n{-}a{+}7\rbrack $  \\
\bottomrule
\end{tabular}
\end{center}
 \caption{Values for the proof of Theorem~\ref{main}.} \label{values}
\end{sidewaystable}

\item After the first two steps we get the triple $(0,1,0)$. Please note that the sets in cases I. and III. are identical, so we can ignore case III. In the first case we have
\begin{small}
\begin{align*}
0=&-3\overbrace{\#_2\lbrack 3n{-}b,3n{-}a\rbrack }^{8n+2}+2\overbrace{\#_2\lbrack n{+}a,n{+}b\rbrack }^{4n+2}+\overbrace{\#_2\lbrack n{-}b,n{-}a\rbrack }^{2n+1}\\
&-\overbrace{\#_2\lbrack n{+}a,n{+}b{+}1\rbrack }^{4n+3}+\overbrace{\#_2\lbrack 5n{+}a,5n{+}b{+}1\rbrack }^{16n+3}\\
&+\overbrace{\#_2\lbrack 5n{+}a{+}1,5n{+}b{+}2\rbrack }^{16n+6}-\overbrace{\#_2\lbrack 3n{-}b,3n{-}a{+}1\rbrack }^{8n+3}
\end{align*}
\end{small}
where all terms in a single line normalize to zero.

In the second case we have
\begin{small}
\begin{align*}
0=&+\overbrace{\#_2\lbrack 5n{+}a{+}2,5n{+}b{+}2\rbrack }^{16n+6}+\overbrace{\#_2\lbrack n{-}b,n{-}a\rbrack }^{2n+1}-2\overbrace{\#_2\lbrack 3n{-}b,3n{-}a\rbrack }^{8n+2}\\
&-\overbrace{\#_2\lbrack 3n{-}b,3n{-}a\rbrack }^{8n+2}+2\overbrace{\#_2\lbrack n{+}a{+}1,n{+}b\rbrack }^{4n+2}-\overbrace{\#_2\lbrack 3n{-}b,3n{-}a{+}1\rbrack }^{8n+3}\\
&-\overbrace{\#_2\lbrack n{+}a{+}1,n{+}b{+}1\rbrack }^{4n+3}+\overbrace{\#_2\lbrack 5n{+}a,5n{+}b{+}1\rbrack }^{16n+3}+1
\end{align*}
\end{small}
where all terms in the first line normalize to zero. We then split off Iverson brackets to get
\begin{small}
\begin{align*}
0=&-\#_2\lbrack 3n{-}b,3n{-}a\rbrack +\#_2\lbrack 5n{+}a,5n{+}b\rbrack +\llbracket 5n{+}b{+}1 \text{ even} \rrbracket\\
&-\#_2\lbrack n{+}a{+}1,n{+}b\rbrack -\llbracket n{+}b{+}1 \text{ even} \rrbracket+\#_2\lbrack n{+}a{+}1,n{+}b\rbrack \\
&-\#_2\lbrack 3n{-}b,3n{-}a{+}1\rbrack +\#_2\lbrack n{+}a{+}1,n{+}b\rbrack +1.
\end{align*}
\end{small}
The first two lines cancel. We split the first interval two times and get
\begin{small}
\[-\#_2\lbrack 3n{-}b,3n{-}a{-}1\rbrack  -\llbracket 3n{-}a \text{ even} \rrbracket-\llbracket 3n{-}a{+}1 \text{ even} \rrbracket +\#_2\lbrack n{+}a{+}1,n{+}b\rbrack +1,\]
\end{small}
and finally
\begin{small}
\[
0=-\llbracket 3n{-}a \text{ even} \rrbracket-\llbracket 3n{-}a{+}1 \text{ even} \rrbracket+1.
\]
\end{small}

\item We already know that the Relation 5. is true so we subtract it from Relation 6. to get:
\[\mathcal{P}_{16n+7}-\mathcal{P}_{16n+6}=\mathcal{P}_{4n+3}-\mathcal{P}_{4n+2}.\]
Now we follow the usual procedure.  The triple is $(0,0,0,)$ and we have to split intervals before we can normalize to zero.

\item Straightforward.  The triple is $(0,0,0,)$ and we can normalize to zero.
\item We subtract Relation 7. from Relation 8. we get $\mathcal{P}_{16n+10}-\mathcal{P}_{16n+8}=0$. Again  the triple is $(0,0,0,)$ and we can normalize to zero.
\item If we subtract Relation 8. from Relation 9. we get:
\[\mathcal{P}_{16n+11}-\mathcal{P}_{16n+10}=-\mathcal{P}_{16n+3}+2\mathcal{P}_{8n+2}-\mathcal{P}_{2n+1}.\]
Our triple is $(-1,-1,-1)$, we have to split  and in all three cases we get
\[0=\llbracket 5n{+}b{+}4 \text{ even} \rrbracket+\llbracket 5n{+}b{+}1 \text{ even} \rrbracket-1 \]
as final result.
\item The triple is $(0,-1,0)$. The rest of the calculation is lengthy, since for the first time all three cases are  different, but not too hard. We have to split
in the second and third case.
\item Another long calculation. Our triple is $(0,-2,-1)$.  We have to split
in the second and third case. We show as an example the fourth step in the second case:
\begin{small}
\begin{align*}
0=&-2\overbrace{\#_2\lbrack 3n{-}b{+}2,3n{-}a{+}2\rbrack }^{8n+7}+2\overbrace{\#_2\lbrack n{-}b,n{-}a\rbrack }^{2n+1}\\
&+3\overbrace{\#_2\lbrack 3n{-}b{+}2,3n{-}a{+}1\rbrack }^{8n+6}-3\overbrace{\#_2\lbrack n{+}a{+}1,n{+}b\rbrack }^{4n+2}\\
&+2\overbrace{\#_2\lbrack 3n{-}b,3n{-}a{+}1\rbrack }^{8n+3}-3\overbrace{\#_2\lbrack n{+}a{+}1,n{+}b\rbrack }^{4n+2}+\overbrace{\#_2\lbrack 5n{+}a{+}5,5n{+}b{+}5\rbrack }^{16n+15}\\
&+\overbrace{\#_2\lbrack n{-}b,n{-}a\rbrack }^{2n+1}-\overbrace{\#_2\lbrack 5n{+}a,5n{+}b{+}1\rbrack }^{16n+3}-2
\end{align*}
\end{small}
The first two lines cancels and we can start to split intervals.
\begin{small}
\begin{align*}
0=&\#_2\lbrack 5n{+}a{+}5,5n{+}b{+}5\rbrack -\#_2\lbrack 5n{+}a{+}1,5n{+}b{+}1\rbrack -\llbracket 5n{+}a\text{ even} \rrbracket\\
&+\#_2\lbrack n{-}b,n{-}a{-}1\rbrack +\llbracket n{-}a\text{ even} \rrbracket-\#_2\lbrack n{+}a{+}1,n{+}b\rbrack \\
&+2\#_2\lbrack 3n{-}b,3n{-}a{+}1\rbrack -2\#_2\lbrack n{+}a{+}1;n{+}b\rbrack -2
\end{align*}
\end{small}
The first two lines cancel and after splitting the first interval in the last line two times we get the true equation
\begin{small}
\[0=+2\llbracket 3n{-}a{+}1\text{ even} \rrbracket+2\llbracket 3n{-}a\text{ even} \rrbracket-2.\]

\end{small}
\item  Straightforward.  The triple is $(0,0,0)$ and we can normalize to zero.
\item We subtract Relation 6. from Relation 13. to get
\[\mathcal{P}_{32n+19}-\mathcal{P}_{16n+7}=0.\]
The triple is $(0,0,0)$. We normalize to zero.
\end{enumerate}
\end{proof}

\section{Properties of $\mathcal{P}_n$}\label{eigen}

In this section we show three additional properties of  $\mathcal{P}_n$. In the first lemma we show that $\mathcal{P}_{n}$ changes in steps of $2$ (for $n>0$).

\begin{lem}\label{dist} For $n\geq 4$ we have
 \[\mathcal{P}_{n+1}-\mathcal{P}_{n}\in \{-2,0,2\}.\]
\end{lem}

\begin{proof}
 Due to Equation (\ref{eq1}) we know that there are three possibilities for the relative sizes of $\mathbf{PAIRS}(2n-2)$ and $\mathbf{PAIRS}(2n)$:
 
 \vspace{3pt}
 \begin{tabular}[h]{rlrlrl}
\multirow{2}{*}{ \hspace{5pt}  1.}&$\lbrack a,b\rbrack $&\multirow{2}{*}{  \hspace{5pt} 2.}&$\lbrack a,b\rbrack $&\multirow{2}{*}{ \hspace{5pt}  3.}&$\lbrack a,b\rbrack $\\
&$\lbrack a+1,b\rbrack $&&$\lbrack a,b+1\rbrack $&&$\lbrack a+1,b+1\rbrack $\\
\end{tabular}
 \vspace{3pt}
 
We use Lemma~\ref{conclass} and with $\mathcal{C}:=\#(\mathbf{PAIRS}(2n)\cap\mathbf{PAIRS}(2n-2))$ and $\mathcal{E}:=\#_2(\mathbf{PAIRS}(2n)\cap\mathbf{PAIRS}(2n-2))$
we get

\begin{center}
 \setlength{\tabcolsep}{2.65pt}
\begin{tabular}[h]{cccc}
 \toprule
Case & $\mathcal{P}_{2n-1}$ &$\mathcal{P}_{2n}$&$\mathcal{P}_{2n+1}$\\
\midrule
$1.$&$2(2(\mathcal{C}{+}1){-}\mathcal{E}{-}\llbracket  a \text{ even} \rrbracket)$&$2((\mathcal{C}{+}1){+}\mathcal{C}{-}\mathcal{E})$&$2(2\mathcal{C}{-}\mathcal{E})$\\
$2.$&$2(2\mathcal{C}{-}\mathcal{E})$&$2((\mathcal{C}{+}1){+}\mathcal{C}{-}\mathcal{E})$&$2(2(\mathcal{C}{+}1){-}\mathcal{E}{-}\llbracket  b{+}1 \text{ even} \rrbracket)$\\
$3.$&$2(2(\mathcal{C}{+}1){-}\mathcal{E}{-}\llbracket  a \text{ even} \rrbracket)$&$2((\mathcal{C}{+}1){+}(\mathcal{C}{+}1){-}\mathcal{E})$&$2(2(\mathcal{C}{+}1){-}\mathcal{E}{-}\llbracket  b{+}1 \text{ even} \rrbracket)$\\
\bottomrule
\end{tabular}
\end{center}
 In all three cases, regardless if $a$ and $b+1$ are even or odd, Lemma~\ref{dist} is true.
\end{proof}

Now we want to prove that the sequence $\mathcal{P}_n$ is unbounded.

\begin{lem} \label{unbo}
If $\mathbf{pairs}(n)=\lbrack a,b\rbrack $, $\mathbf{pairs}(n+1)=\lbrack a,b+1\rbrack $ with $n$ and $b$ odd and $a$ even and  then  the sequence $a_0=n$, $a_{i+1}=16a_i-5$ satisfies
\[\mathcal{P}_{a_n}=\mathcal{P}_{a_0}+6n.\]
\end{lem}
\begin{proof}
If we start with two  sets $\mathbf{pairs}(n)=\lbrack a,b\rbrack $ and $\mathbf{pairs}(n+1)=\lbrack a,b+1\rbrack $ we can use  Theorem~\ref{pairs} to get
$\mathbf{pairs}(16n-5)=\lbrack 5n+a-3,5n+b-1\rbrack $ and $\mathbf{pairs}(16n-6)=\lbrack 5n+a-3,5n+b\rbrack $.

So if we start in case III. with $\mathbf{pairs}(n)$ and $\mathbf{pairs}(n+1)$ we will be again in case III. with $\mathbf{pairs}(16n-5)$ and $\mathbf{pairs}(16n-6)$.
Therefore we can concatenate the whole process.

Furthermore, if  $n$ and $b$ are odd and $a$ is even then $16n-5$ and $ 5n+b-1$ are odd and $5n+a-3$ is even. Hence the sets $\#\mathbf{pairs}(n)=\#\lbrack a,b\rbrack $ and $\#\mathbf{pairs}(16n-5)=\#\lbrack a',b'\rbrack $
contain an even number of elements and we have:
\[\#\lbrack a,b\rbrack =2s,\enskip\#\lbrack a',b'\rbrack =2s+2,\enskip\#_2\lbrack a,b\rbrack =s,\enskip\#_2\lbrack a',b'\rbrack =s+1.\]
Since the values $n$ and $16n-5$ are both odd  we know from Theorem~\ref{pairs} that $\mathbf{PAIRS}(n-1)=\mathbf{pairs}(n)$ and $\mathbf{PAIRS}(16n-6)=\mathbf{pairs}(16n-5)$. Now we
apply Lemma~\ref{conclass} and get $\mathcal{P}_n=6s$ and $\mathcal{P}_{16n-5}=6s+6$.
\end{proof}
The sequence $\mathcal{P}_3=6,\mathcal{P}_{43}=12$, $\mathcal{P}_{683}=18,\dots$ is one example of such an unbounded sequence.
In \cite{Karhumaeki2013} they show $\mathcal{P}_t^{(2)}((2\cdot 4^m+4)/3)=\Theta(m)$ which also proves that the sequence $\mathcal{P}_n$ is unbounded.

The last theorem reveals the symmetries  of $\mathcal{P}_n$. 
\begin{table}[h]
\begin{footnotesize}
\begin{center}
\begin{tabular}[h]{ccr @{\hskip 1pt} l@{\hskip -10pt}cr @{\hskip 1pt} l@{\hskip -9pt}c@{\hskip -6pt}cc}
\toprule
Case & $i$ & \multicolumn{2}{c}{$\mathbf{pairs}(n{+}i)$} &  &  \multicolumn{2}{c}{$\mathbf{pairs}(2n{+}i)$}&\hspace{13pt}$\#\mathbf{pairs}(n{+}i)$&&\hspace{13pt}$\#\mathbf{pairs}(2n{+}i)$\\
\midrule
 \multirow{3}{*}{I.} & $\scriptstyle{-}1$  && & \multirow{3}{*}{$\longmapsto$}&$\lbrack a'{+}1$&$, b'\rbrack $&& \multirow{3}{*}{$\longmapsto$}&$\mathcal{C}$\\
    & $\scriptstyle0$&$\lbrack a$&$, b\rbrack $ &&$\lbrack a'{+}1$&$,b'{+}1\rbrack $&$\mathcal{C}$&&$\mathcal{C}+1$\\
    &  $\scriptstyle{+}1$&$\lbrack a$&$, b\rbrack $ & &$\lbrack a'$&$,  b'{+}1\rbrack $&$\mathcal{C}$&&$\mathcal{C}$ \\
 \midrule
   \multirow{3}{*}{II.} & $\scriptstyle{-}1$  &&& \multirow{3}{*}{$\longmapsto$}&$\lbrack a'{+}1$&$,  b'\rbrack $&&\multirow{3}{*}{$\longmapsto$}&$\mathcal{C}+1$\\
    & $\scriptstyle0$&$\lbrack a{+}1$&$, b\rbrack $ &&$\lbrack a'{+}1$&$, b'\rbrack $&$\mathcal{C}+1$&&$\mathcal{C}+1$\\
    &  $\scriptstyle{+}1$&$\lbrack a$&$,  b\rbrack $ &&$\lbrack a'$&$, b'\rbrack $&$\mathcal{C}$&&$\mathcal{C}$\\
 \midrule
 \multirow{3}{*}{III.}  & $\scriptstyle{-}1$  &&& \multirow{3}{*}{$\longmapsto$}&$\lbrack a'$&$, b'\rbrack $&&\multirow{3}{*}{$\longmapsto$}&$\mathcal{C}$\\
    & $\scriptstyle0$&$\lbrack a$&$, b\rbrack $ &&$\lbrack a'$&$, b'{+}1\rbrack $&$\mathcal{C}$&&$\mathcal{C}+1$\\
    &  $\scriptstyle{+}1$&$\lbrack a$&$, b{+}1\rbrack $ &&$\lbrack a'$&$, b'{+}1\rbrack $&$\mathcal{C}+1$&&$\mathcal{C}+1$\\
\bottomrule
\end{tabular}
\end{center}
\end{footnotesize}
 \caption{The action of Theorem~\ref{pairs} on $\mathbf{pairs}(n)$ and $\mathbf{pairs}(n+1)$}\label{set}
\end{table}
\begin{theorem}\label{pal}
 The $2$-abelian complexity of the Thue--Morse word is a concatenation of palindromes of increasing size since the sequence
 \[\mathcal{P}_{2^q+1}\mathcal{P}_{2^q+2}\mathcal{P}_{2^q+3}\cdots\mathcal{P}_{2^{q+1}+1}\]
 is a palindrome, or equivalently
 \[\mathcal{P}_{2^q+1+i}=\mathcal{P}_{2^{q+1}+1-i}\]
for  $0\leq i \leq 2^{q-1}$.
\end{theorem}
\begin{proof}
First we will show by induction that the sequence
\[\#\mathbf{pairs}(2^q+1),\#\mathbf{pairs}(2^q+2),\dots,\#\mathbf{pairs}(2^{q+1}+1)\]
is a palindrome. The base case is
\[\#\mathbf{pairs}(3)=2, \#\mathbf{pairs}(4)=3, \#\mathbf{pairs}(5)=2.\]

Two sets  $\mathbf{pairs}(2^q+1+ i)$ and $\mathbf{pairs}(2^{q+1}+1-i)$ with $0\leq i \leq 2^{q-1}$ are called \emph{corresponding sets}. 
If a consecutive pair of sets is mapped to a consecutive triple of sets
\[\begin{matrix}\\\mathbf{pairs}(2^q+i\hspace{17pt})\\\mathbf{pairs}(2^q+i+1)\end{matrix}\mapsto\begin{matrix}\mathbf{pairs}(2^{q+1}+2i-1)\\\mathbf{pairs}(2^{q+1}+2i\hspace{17pt})\\\mathbf{pairs}(2^{q+1}+2i+1)\end{matrix}\]
with $1\leq i \leq 3\cdot2^{q-1}$ then the corresponding  pair of consecutive sets is mapped to a consecutive triple of sets
\[\begin{matrix}\\\mathbf{pairs}(2^{q+1}-i+1)\\\mathbf{pairs}(2^{q+1}-i+2)\end{matrix}\mapsto\begin{matrix}\mathbf{pairs}(2^{q+2}-2i+1)\\\mathbf{pairs}(2^{q+2}-2i+2)\\\mathbf{pairs}(2^{q+2}-2i+3)\end{matrix}.\]

Now we look at Table~\ref{set} and see that it is enough to know the relative sizes of consecutive pairs of sets do determine in which case we are. So if a 
consecutive  pair of sets is mapped to a consecutive  triple of sets via case II. the corresponding  consecutive pair of sets   is
mapped to a corresponding consecutive triple of sets via case III. and vise versa. If we have a case I. map for the consecutive pair of sets  we also have a case I. map for the
corresponding consecutive   
 pair of sets. In all three cases the palindromic structure of the set cardinality is preserved.

Now we show that \[\#_2\mathbf{pairs}(2^q+1),\#_2\mathbf{pairs}(2^q+2),\dots,\#_2\mathbf{pairs}(2^{q+1}+1)\]
is a palindrome too. We do this by showing that for two corresponding sets  $\mathbf{pairs}(2^q+1+ i)=\lbrack a,b\rbrack $ and $\mathbf{pairs}(2^{q+1}+1-i)=\lbrack a',b'\rbrack $ we have
\begin{equation} \label{eq}
 a\equiv b' \pmod 2 \text{ and }a'\equiv b \pmod 2. 
\end{equation}
 Since two corresponding sets have the same cardinality we can conclude that $\#_2\lbrack a,b\rbrack =\#_2\lbrack a',b'\rbrack $.

 Since $\mathcal{P}_3=6$, Relation 1. tells us $\mathcal{P}_{2^q+1}=6$,
so Equation (\ref{eq}) is true for our base case,  the corresponding sets
$\mathbf{pairs}(2^q+1)$ and $\mathbf{pairs}(2^{q+1}+1)$.

We use induction to go from $\mathbf{pairs}(2^q+1+ i)$ and $\mathbf{pairs}(2^{q+1}+1-i)$ to
$\mathbf{pairs}(2^q+1+ (i+1))$ and $\mathbf{pairs}(2^{q+1}+1-(i+1))$. We have to check the three cases from Table~\ref{set} again:

In the first case nothing changes, we go from $\lbrack a,b\rbrack $ and $\lbrack a',b'\rbrack $ to $\lbrack a,b\rbrack $ and $\lbrack a',b'\rbrack $ and Equation~\ref{eq} is trivially fulfilled.

In the second case we go from $\lbrack a,b\rbrack $ and $\lbrack a',b'\rbrack $ to $\lbrack a+1,b\rbrack $ and $\lbrack a',b'-1\rbrack $ and Equation~\ref{eq} is fulfilled again.

In the third case the step is from $\lbrack a,b\rbrack $ and $\lbrack a',b'\rbrack $ to $\lbrack a,b+1\rbrack $ and $\lbrack a'-1;b'\rbrack $ and Equation~\ref{eq} holds.

Since \[\#\mathbf{pairs}(2^q+1+ i)=\#\mathbf{pairs}(2^{q+1}+1-i)\]
and also
\[\#_2\mathbf{pairs}(2^q+1+ i)=\#_2\mathbf{pairs}(2^{q+1}+1-i)\]
with $0\leq i \leq 2^{q-1}$ we can use Theorem~\ref{pairs} and to get
 \[\#\mathbf{PAIRS}(2^q+ 2i)=\#\mathbf{PAIRS}(2^{q+1}-2i)\]
and also
\[\#_2\mathbf{PAIRS}(2^q+ 2i)=\#_2\mathbf{PAIRS}(2^{q+1}-2i)\]
with $0\leq i \leq 2^{q-2}$. Now we use Lemma~\ref{conclass} and get
$\mathcal{P}_{2^q+1+i}=\mathcal{P}_{2^{q+1}+1-i}$ with $0\leq i \leq 2^{q-1}$.
\end{proof}

\section*{Acknowledgements}

The author is supported by the Austrian Science Fund FWF projects 
 W1230, Doctoral Program ``Discrete Mathematics'', and F5503 (part of the special research program (SFB) ``Quasi-Monte Carlo Methods: Theory and Applications'').
 The author is grateful to his advisor Peter Grabner and a anonymous referee whose comments and feedback greatly improved this article.
\bibliographystyle{abbrv}
\bibliography{Literatur.bib}
\end{document}